\documentclass{amsart}

\usepackage{amsmath, mathrsfs, amssymb,amsthm,hyperref,tikz,tikz-3dplot,centernot,comment,stmaryrd,enumerate,mathtools}
\usepackage{tikz-cd}
\hypersetup{pdfstartview={XYZ null null 1.25}}
\usepackage[all]{xy}

\newcommand{\bv}{\bigvee}
\newcommand{\bw}{\bigwedge}

\newcommand{\sL}{\mathscr L}
\newcommand{\cA}{\mathcal A}
\newcommand{\cB}{\mathcal B}
\newcommand{\cC}{\mathcal C}
\newcommand{\cG}{\mathcal G}

\newcommand{\ra}{\rightarrow}

\newcommand{\join}{\mathsf{j}}
\newcommand{\meet}{\mathsf{m}}

\theoremstyle{plain}
\newtheorem{thm}{Theorem}[section]
\newtheorem{prop}[thm]{Proposition}
\newtheorem{cor}[thm]{Corollary}
\newtheorem{lemma}[thm]{Lemma}
\newtheorem{ex}[thm]{Example}
\theoremstyle{definition}
\newtheorem{defn}[thm]{Definition}

\title{Recursive axiomatisations from separation properties}
\author{Rob Egrot}
\date{}
\address{Faculty of ICT, Mahidol University, 999 Phutthamonthon 4 Rd, Salaya, Nakhon Pathom 73170, Thailand}
\email{robert.egr@mahidol.ac.th}
\begin{document}

\begin{abstract}
We define a fragment of monadic infinitary second-order logic corresponding to an abstract separation property. We use this to define the concept of a \emph{separation subclass}. We use model theoretic techniques and games to show that separation subclasses whose axiomatisations are recursively enumerable in our second-order fragment can also be recursively axiomatised in their original first-order language. We pin down the expressive power of this formalism with respect to first-order logic, and investigate some questions relating to decidability and computational complexity. As applications of these results, by showing that certain classes can be straightforwardly defined as separation subclasses, we obtain first-order axiomatisability results for these classes. In particular we apply this technique to graph colourings and a class of partial algebras arising from separation logic.
\end{abstract}
\maketitle

\section{Introduction}
We begin with a motivating example. Precise definitions will be given in the next section. A partially ordered set (poset) is \emph{representable} if it can be embedded into a powerset algebra via a map that preserves existing finite meets and joins. The class of representable posets and its infinitary variations have been studied, not always using this terminology, in \cite{CheKem92,Kem93,HicMon84,VanA16,Egr16,Egr17a,Egr17b,Egr18,Egr19}, generalising work done in the setting of semilattices \cite{Bal69,Sch72,CorHic78,Kea97}, and for distributive lattices and Boolean algebras \cite{Bir33,Stone36,Loom47,Sik48,BirFri48,Cha57,ChaHor62,Abi71,EgrHir12}. At first glance, it is far from obvious that the class of representable posets is elementary. However, it is fairly easy to show that a poset is representable if and only if it has a `separating' set of `prime filters'. More precisely, a poset $P$ is representable if and only if whenever $p\not\leq q\in P$ there is a `prime filter' of $P$ containing $p$ and not $q$. Note that there are several non-equivalent concepts of `prime filter of a poset' in circulation, and we are using one in particular. A more precise definition is given in Example \ref{E:poset1}.

Now, given the description of the representable posets in terms of this `separation property', it is possible to show that it can in fact be axiomatised in first-order logic. \cite[Theorem 4.5]{Egr16} does this by proving closure under taking isomorphisms, ultraproducts and ultraroots and appealing to the Keisler-Shelah theorem \cite{Kei61,She71}, and similar can be done by proving closure under taking ultraproducts and elementary substructures and appealing to \cite[Theorem 2.13]{FMS62}. Such a non-constructive proof of existence may be regarded as being of limited practical use, however, the very fact that an axiomatisation is known to exist can be used in a neat trick to show that a certain constructively generated axiomatisation is correct. This is the main result of \cite{Egr19}.

The method of \cite{Egr19}, which is not novel, is to describe the `separation property' of representable posets in terms of a game played between two players. The game is defined so that the number of rounds a certain player can survive in a particular game corresponds, in a sense, to how close a given poset is to being representable. First-order axioms are then written down that correspond to the player `having a strategy' in a game. These axioms are shown to correctly axiomatise the class of representable posets by means of the `neat trick' mentioned previously. 

A similar idea appears in \cite{HirMcL17}, where it is used to find an explicit axiomatisation for a certain class of partial algebras of partial functions that appears in connection with separation logic. Again we have a class which is not obviously elementary, but which can fairly easily be shown to be definable in terms of a `separation property'. The separation property is then used to show, non-constructively, that a first-order axiomatisation exists, and then to construct explicit axioms based on games which are, using the `neat trick', shown to be an axiomatisation for the class.  

The main purpose of this paper is to prove a general theorem that includes the relevant results of \cite{HirMcL17,Egr19} as special cases, and is also applicable in a wide variety of other situations. The strategy is to first formalise the concept of a `separation property' in a way that allows the necessary results to go through, while also being intuitive enough to be useful in practice. This is done in Section \ref{S:sep}. In particular, the basic definition of a \emph{separation subclass} is made. The sense in which separation subclasses can be, for example, \emph{essentially countable}, or \emph{essentially recursively enumerable}, is also explained.

We formalise the concept of a separation subclass using infinitary monadic second-order logic. We show that if $\cA$ is a class of structures and $\cB$ is a subclass of $\cA$ that is elementary relative to $\cA$, then $\cB$ can always be described as separation subclass of $\cA$ (Proposition \ref{P:AxSep}). More interestingly, we show that every separation subclass of an elementary class has a first-order axiomatisation relative to the superclass (Theorem \ref{T:same}). Thus the classes of separation subclasses and elementary subclasses of an elementary class coincide. However, the important difference is that descriptions as separation subclasses can often be much easier to find than first-order axiomatisations. Moreover, as we shall see, provided the superclass is elementary, we can use a description of a subclass as an essentially recursively enumerable separation subclass to automate the construction of an explicit first-order axiomatisation.     

In Section \ref{S:game} we describe a class of games played between two players, $\forall$ and $\exists$. The key result is that, if $\cB$ is an essentially countable separation subclass of $\cA$, then given $A\in \cB$, the player $\exists$ has a strategy for never losing in every relevant game. Conversely, if $A\in\cA$ is countable, then $\exists$ having such strategies implies that $A\in\cB$ (Proposition \ref{P:game}). 

Section \ref{S:ax} formalises the existence of strategies for $\exists$ in first-order logic. The main result, which is stated as Corollary \ref{C:main}, is that an essentially recursively enumerable separation subclass $\cB$ of an elementary class $\cA$ always has a recursive first-order axiomatisation relative to $\cA$, which we can generate systematically by examining the relevant class of games. Moreover, we present simple sufficient conditions for the axiomatisation produced to be universal.

In Section \ref{S:decide} we collect together some previous results to make explicit the connections between the various kinds of separation subclasses and the various ways a class can be elementary relative to its superclass (Proposition \ref{P:LclassesD}). We also make some simple observations regarding decision problems and complexity (Propositions \ref{P:semi} and \ref{P:CclassesD}). 

Finally, in Section \ref{S:Apps} we present some applications of the general theory we have developed. First we show how the work in \cite{HirMcL17} on disjoint union partial algebras fits into the framework of separation subclasses, and how this automatically proves some of the results of that paper  (Section \ref{S:dupa}). Following this we consider graph colourings. In particular, in Section \ref{S:N-colour}, from the fact that the class of $N$-colourable graphs has a simple description as a separation subclass of the class of all graphs, we are able to find easy proofs of several model theoretic results relating to these structures. We present new proofs of the known results that, for all $N\geq 2$, the class of $N$-colourable graphs has a universal Horn axiomatisation, but is not finitely axiomatisable, and also that, when $N\geq 3$, the class of graphs with chromatic number $N$ is not elementary. The proofs follow directly from the general results on separation subclasses. In this sense, $N$-colourable graphs provide a good example of a class where a characterisation as a separation subclass is obvious, but where results relating to first-order axiomatisability are not so obvious. 

In Sections \ref{S:clique} and \ref{S:harm} we describe the classes of graphs with $N$-clique covers, and harmonious $N$-colourings, respectively, as separation subclasses. Thus, as an immediate consequence, we can show that both classes have recursive universal axiomatisations relative to the class of all graphs. Moreover, our method proves that the class of graphs with harmonious $N$-colourings is actually finitely axiomatisable for each $N\geq 1$.

\section{Separation subclasses}\label{S:sep}
We adopt the convention that indexing sets are denoted by capital letters, and arbitrary indices taken from these sets use the corresponding lowercase letters. When dealing with a set $\{x_1,\ldots,x_N\}$, we will use $x_n$ to denote an arbitrary element from this set. For variable symbols, we adopt the convention that e.g. $\vec{x}_N$ denotes the set $\{x_1,\ldots,x_N\}$. Given a first-order formula $\phi$ and a set of variable symbols $\vec{x}_N$, it is common to write something like $\phi(\vec{x}_N)$ to denote that the free variables of $\phi$ are precisely $\vec{x}_N$. In this paper we use a relaxed version of this convention. Here we will write e.g. $\phi(\vec{x}_N)$ to denote that the free variables of $\phi$ are from among $\vec{x}_N$, but are not necessarily all of them. We will write things like $\forall \vec{x}_N \phi(\vec{x}_N)$ to stand for $\forall x_1\ldots x_N \phi(\vec{x}_N)$. Given a set of variables $X$ we may also write e.g. $\phi(X)$ to denote that the free variables of $\phi$ are from $X$.

We will also use the $\vec{\phantom{x}}$ notation in a different but closely related way as follows. If e.g. $b_1,\ldots,b_N$ are not necessarily distinct elements of some $\sL$-structure $\cB$, we will write $\phi(\vec{b}_N)$ to mean $\phi(\vec{x}_N)$ where each free $x_n$ is interpreted as $b_n$ for all $n\in\{1,\ldots,N\}$. We will sometimes write something like $\vec{b}_N\in \cB$ to denote a sequence $b_1,\ldots,b_N$ of elements of a structure $\cB$.

Before diving further into the technicalities, we will try to build up some intuition for what is to be done. The situation to be captured is as follows. We have some kind of a structure, e.g. a poset $P$, and we want to say that given some elements of this structure collectively meeting some condition that is definable in first-order logic, e.g. $p,q\in P$ with $p\not\leq q$, there are subsets of the original structure which can be specified to either contain or not contain the elements in question. For, example given $p\not\leq q\in P$ we may want to demand that there is a prime filter of $P$ containing $p$ and not $q$. Moreover, these sets have to satisfy first-order closure conditions. For example, prime filters must be closed upwards, among other things. It is convenient to represent these sets using new monadic predicate symbols. We will ultimately want to existentially quantify over these predicate symbols, so we consider them to be second-order variables. If $C$ is the monadic predicate standing for the prime filter in our poset example, then upward closure can be expressed using the sentence 
\[\tag{$\dagger$} \forall x\forall y\big((x\leq y \wedge C(x)) \ra C(y)\big).\]
The idea is that the relationship between elements, in this case picked out by the variables $x$ and $y$, puts some constraint on their containment or otherwise in the set, in this case picked out by the second-order variable $C$. It will be useful to rewrite this so that the purely first-order part, i.e. the part not involving $C$, is the antecedent of an implication, with the part involving $C$ being the consequent. In other words, to clearly reflect the fact that the first-order relationship between the elements places constraints on the placements of the variables in the sets represented by the second-order variables. These constraints on set containment can be expressed by Boolean combinations of clauses demanding that particular elements either are or are not included in a particular set. From $(\dagger)$, for example, we can obtain the equivalent formula 
\[\tag{$\ddagger$}\forall x\forall y\big(x\leq y \ra (\neg C(x) \vee C(y))\big).\]    
What we call a \emph{closure rule}, and define explicitly in Definition \ref{D:cRule} shortly, is just a conjunction of constraints of this form for some given sets. For example, we describe the closure properties of prime filters by conjoining $(\ddagger)$ with an infinite number of other sentences of similar form (see Example \ref{E:poset1} below for the details).

In  Definition \ref{D:Srule} we will formally introduce what we call a \emph{separation rule}. The intuition here is that a separation rule expresses that if a set of elements from the structure meet some first-order definable condition, then there exist some subsets of the structure such that 1) there are constraints on the containment of the original elements in these sets, and 2) these sets satisfy a closure rule. Given a poset $P$, for example, if given a monadic predicate $C$ we use $C[P]$ to denote $\{p\in P:C(p)\}$, the idea that given $p\not\leq q\in P$ there is a prime filter containing $p$ and not $q$ can be expressed as a separation rule as follows
\[\forall p, q\big(p\not\leq q \ra \exists C(C(p)\wedge \neg C(q) \wedge \text{ `closure rule expressing that $C[P]$ is a prime filter'})\big).\]

\begin{defn}\label{D:cRule}
Let $\sL$ be a first-order signature, let $1\leq K\in\omega$, and let $C_1,\ldots,C_K$ be unary predicate symbols not appearing in $\sL$.  Define $\sL^+_{\vec{C}_K}=\sL\cup\{C_1,\ldots,C_K\}$. A $\vec{C}_K$-\textbf{closure rule} is a conjunction $\bw_I \tau_i$ for some $I\neq\emptyset$, where for each $i\in I$, the formula $\tau_i$ is an $\sL^+_{\vec{C}_K}$-sentence of form  
\[\forall \vec{y}_{M_i}\Big(\gamma(\vec{y}_{M_i})\rightarrow \psi(\vec{y}_{M_i})\Big),\]
where $\gamma$ is a first-order $\sL$-formula with free variables taken from $\vec{y}_{M_i}=\{y_1,\ldots,y_{M_i}\}$, and $\psi$ is a quantifier-free first-order $\sL^+_{\vec{C}_K}$-formula whose free variables are also taken from $\{y_1,\ldots,y_{M_i}\}$. Note that $I$ may be infinite. 
\end{defn}

\begin{defn}\label{D:Srule}
A \textbf{separation rule} for a first-order signature $\sL$ is anything falling into either of the following two categories:
\begin{enumerate}[(1)] 
\item $\sL$-sentences.
\item Monadic second-order sentences of form 
\[\forall\vec{x}_N\Big(\mu(\vec{x}_N)\rightarrow \exists \vec{C}_K(\eta(\vec{x}_N)\wedge \tau)\Big),\] 
where $1\leq K<\omega$, where $\mu$ is an $\sL$-formula with free variables taken from $\vec{x}_N = \{x_1,\ldots,x_N\}$, where $\eta$ is a quantifier-free $\sL^+_{\vec{C}_K}$-formula whose free variables are also taken from $\vec{x}_N$, and where $\tau$ is either a $\vec{C}_K$-closure rule or the tautology $\top$. 

So, in the case where $\tau\neq\top$, a separation rule has form 
\[\forall\vec{x}_N\Big(\mu(\vec{x}_N)\rightarrow \exists \vec{C}_K\big(\eta(\vec{x}_N)\wedge \bw_I\forall \vec{y}_{M_i}(\gamma_i(\vec{y}_{M_i})\rightarrow \psi_i(\vec{y}_{M_i}))\big)\Big)\] 
for some set $I\neq\emptyset$ (which may be infinite). 
\end{enumerate} 

The \textbf{order} of a separation rule of type (1) is said to be zero, and the order of a separation rule of type (2) is the value of $K$ used in its definition.

A separation rule $\sigma$ is said to be \emph{finite} if it either has order zero, or if the order is positive with $\sigma = \forall\vec{x}_N\Big(\mu(\vec{x}_N)\rightarrow \exists \vec{C}_K(\eta(\vec{x}_N)\wedge \bw_I\tau_i)\Big)$ and $I$ is finite. A separation rule that is not finite is said to be \emph{infinite}. Similarly, $\sigma$ is \emph{countable} if it is either finite or $I$ is countable. Finally, $\sigma$ is said to be \emph{recursively enumerable} (r.e.) if it is either finite or there is an algorithm for listing the formulas $\tau_i$. In other words, while infinite separation rules cannot be be written out in full, if they are r.e. they can at least be approximated to arbitrary precision by including more of the infinite conjunction $\bw_I\tau_i$.

 A set $\Sigma$ of separation rules is called a \textbf{separation scheme}. Informally, a separation scheme is just a set of constraints on a structure that can be expressed as separation rules. A separation scheme $\Sigma$ is said to be \textbf{essentially finite} if $\Sigma$ is finite and each $\sigma\in \Sigma$ is also finite. Similarly, $\Sigma$ is \textbf{essentially countable} if it is countable and each $\sigma\in \Sigma$ is countable. A separation scheme $\Sigma$ is said to be \textbf{essentially recursively enumerable}  if $\Sigma = \{\sigma_n:n\in \omega\}$, and $\sigma_n$ is r.e. for each $n\in \omega$. 
\end{defn}

\begin{defn}[$\Sigma^{>0}$]
If $\Sigma$ is a separation scheme, we use $\Sigma^{>0}$ to denote the subset of $\Sigma$ containing all the separation rules of order strictly greater than zero (i.e. all those of type (2)).
\end{defn}

\begin{ex}\label{E:poset1}
As mentioned previously, there are several non-equivalent ways to generalise the concept of a prime filter from lattices to posets. For us, a `prime filter' of a poset $P$ is a subset $\Gamma$ of $P$ that is closed upwards, closed under finite existing meets (greatest lower bounds), and has the `primality' property that if the join (least upper bound) of a finite set $S$ of elements of $P$ is defined in $P$ and is in $\Gamma$, then $\Gamma\cap S\neq \emptyset$. We will in future refer to a subset of a poset satisfying these closure properties as an \textbf{$\omega$-filter}, to avoid ambiguity. We want to phrase the condition that, given a poset $P$, if $p\not\leq q\in P$, then there is an $\omega$-filter of $P$ containing $p$ but not $q$ as a separation rule. We proceed as follows.
    
Let $\sL = \{\leq\}$ be the signature of ordered sets. For each $M\geq 1$ let $\vec{y}_M = \{y_1,\ldots,y_M\}$ be a set of variable symbols, and, with a new variable symbol $z$, define $\mathsf{j}_M(\vec{y}_M,z)$ and $\mathsf{m}_M(\vec{y}_M,z)$ to be the universal $\sL$-formulas stating that $z$ is the least upper bound (join) and greatest lower bound (meet) of the elements of $\vec{y}_M$ respectively.  

We define a closure rule to capture the closure properties of $\omega$-filters. To do this we introduce a unary predicate $C$ meant to represent an $\omega$-filter. First we define a clause $\tau_0$ meant to capture upward closure: 
\[\tau_0 = \forall yz \big(y\leq z \ra (C(y)\ra C(z)) \big).\]
 
Now we define a clause $\tau_i$ for each $1\leq i<\omega$ meant to capture closure under finite meets and the `primality' property. We will use even values of $i$ to capture closure under meets of the various finite cardinalities, and we will use odd values to capture the `primality' properties. This division is purely an accounting device, but it is convenient.

If $i = 2M$ for some $M\in\omega$, then
\[\tau_i = \forall \vec{y}_Mz\Big(\meet_M(\vec{y}_M,z) \ra \big (\bw_{m= 1}^M C(y_m)\ra C(z) \big) \Big),\]
 and, if $i = 2M -1$ for some $M\in\omega$, then
\[\tau_i = \forall \vec{y}_Mz\Big(\join_M(\vec{y}_M,z) \ra \big (C(z)\ra \bv_{m= 1}^M C(y_m)) \big) \Big).\]
 The closure rule for $\omega$-filters is then the conjunction $\tau = \bw_{i\in\omega} \tau_i$. We can now define our separation rule $\sigma$ as follows:
\[\sigma = \forall pq \Big(p\not\leq q \ra \exists C \big(C(p)\wedge \neg C(q) \wedge \bw_{i\in\omega} \tau_i \big)\Big).\]
Then $\sigma$ is easily seen to be a r.e. separation rule, and a poset $P$ satisfies $\sigma$ if and only if whenever $p\not\leq q\in P$ there is an $\omega$-filter of $P$ containing $p$ but not $q$, as required. Note that $\sigma$ is not finite as we need $\tau_i$ for all $i\in\omega$.  
\end{ex}

\begin{defn}\label{D:sep}
Let $\sL$ be a first-order signature, let $\cA$ be a class of $\sL$-structures, and let $\cB$ be a subclass of $\cA$. Then $\cB$ is a \textbf{separation subclass} of $\cA$ if there is a separation scheme $\Sigma$ such that $\cB=\{A\in\cA: A\models \Sigma\}$. Here $\models$ is defined using the standard semantics for second-order logic. A separation subclass is said to be \textbf{essentially r.e./countable/finite} when it can be defined using a separation scheme with the corresponding property. If $\cA$ is the class of all $\sL$-structures then we say $\cB$ is a \textbf{separation class}.   
\end{defn}

\begin{ex}\label{E:poset2}
We say a poset is $P$ is \textbf{representable} if there is a set $X$ and an order embedding $h:P\to \wp(X)$ such that $h$ preserves finite meets and joins from $P$ whenever they exist (here $\wp(X)$ is considered as a lattice with operations $\cup$ and $\cap$). It is easy to prove that a poset $P$ is representable if and only if whenever $p\not\leq q\in P$ there is an $\omega$-filter of $P$ containing $p$ and not $q$ (see, for example, \cite[Theorem 2.4]{Egr19}). Thus, building on Example \ref{E:poset1}, we see that the class of representable posets is an essentially r.e. separation subclass of the class of posets, using the separation scheme $\Sigma = \{\sigma\}$. Note that $\Sigma$ is essentially r.e. but not essentially finite, as while it contains only a single separation rule, this separation rule is infinite.

Generalising, given any $2\leq \alpha,\beta\leq\omega$ we say a poset $P$ is $(\alpha,\beta)$-representable if there is a set $X$ and an order embedding $h:P\to \wp(X)$ such that $h$ preserves meets of cardinality strictly less than $\alpha$, and joins of cardinality strictly less than $\beta$. Adapting the previous argument we can show the class of $(\alpha,\beta)$-representable posets is an essentially r.e. separation subclass of the class of all posets, and is essentially finite when $\alpha,\beta<\omega$.
\end{ex}

As may be expected given the definitions, the machinery of separation subclasses is not weaker than the machinery of first-order logic when it comes to specifying subclasses of classes of $\sL$-structures. We make this precise in the following proposition. More surprisingly, it turns out that is not stronger either. This is the result of Theorem \ref{T:same}.  

\begin{prop}\label{P:AxSep}
If $\cA$ is a class of $\sL$-structures, and if $\cB\subseteq \cA$ is elementary relative to $\cA$, then $\cB$ is a separation subclass of $\cA$. Moreover, if the axiomatisation of $\cB$ relative to $\cA$ is finite/countable/r.e., then $\cB$ is an essentially finite/countable/r.e. separation subclass of $\cA$. 
\end{prop}
\begin{proof}
Since separation rules of order zero are just $\sL$-sentences, this is an immediate consequence of Definition \ref{D:sep}.
\end{proof}

It will be useful to slightly generalise the familiar notion of a pseudoelementary class.

\begin{defn}
Let $\cA$ be a class of $\sL$-structures, and let $\cB\subseteq\cA$. Then $\cB$ is \textbf{pseudoelementary relative to} $\cA$ if there is an extension $\sL'$ of $\sL$, and an $\sL'$-theory $T$ such that 
\[\cB=\{A\in \cA : \text{we can interpret the additional symbols of } \sL' \text{ so that } A\models T\}.\]   
\end{defn}

If $\cA$ is the class of all $\sL$-structures, then being pseudoelementary relative to $\cA$ is the same as being pseudoelementary as it is usually defined. We say a pseudoelementary class $\cB$ is \emph{essentially finite/countable/r.e.} relative to $\cA$ if $\sL'$ and $T$ are both finite/countable/r.e.\/ Classes that are essentially finite as pseudoelementary classes relative to the class of all $\sL$-structures are often referred to as being \textbf{basic pseudoelementary}. 

\begin{lemma}\label{L:pseud}
If $\cB$ is a separation subclass of $\cA$, then $\cB$ is pseudoelementary relative to $\cA$. Moreover, if $\cB$ is essentially finite/countable/r.e. as a separation subclass of $\cA$, then $\cB$ is essentially finite/countable/r.e. pseudoelementary relative to $\cA$.
\end{lemma}
\begin{proof}
Let $\cA$ be a class of $\sL$-structures, let $\Sigma$ be a separation scheme defining $\cB$ relative to $\cA$, and let $\sigma\in\Sigma$ have order $K$ for some $K>0$ (as there is nothing to do in the case where $K=0$). So 
\[\sigma = \forall\vec{x}_N\Big(\mu_\sigma(\vec{x}_N)\rightarrow \exists \vec{C}_K\big(\eta_\sigma(\vec{x}_N)\wedge \tau_\sigma \big)\Big),\]
where we either have $\tau_\sigma = \top$ or 
\[\tau_\sigma = \bw_I \tau_\sigma^i.\]
Moreover, assuming $\tau_\sigma\neq \top$, for each $i\in  I$, the formula $\tau_\sigma^i$ is given by
\[\tau_\sigma^i = \forall \vec{y}_{M_i}(\gamma_\sigma^i(\vec{y}_{M_i})\rightarrow \psi_\sigma^i(\vec{y}_{M_i})).\]
Expand $\sL$ to a new signature $\sL'_\sigma$ by adding new $(n+1)$-ary predicate symbols $R_1,\ldots,R_K$. Define
\[t_\sigma = \forall \vec{x}_N(\mu_\sigma(\vec{x}_N)\rightarrow \hat{\eta}_\sigma(\vec{x}_N)),\]
where $\hat{\eta}_\sigma$ is $\eta_\sigma$ but with every occurrence of $C_k(-)$ replaced by $R_k(\vec{x}_N,-)$, for all $k\in\{1,\ldots,K\}$.

Now, assuming $\tau_\sigma\neq \top$, let $i\in I$, and define
\[t_\sigma^{i}= \forall \vec{x}_N\Big(\mu_\sigma(\vec{x}_N)\rightarrow \forall \vec{y}_{M_i}\big(\gamma_\sigma^i(\vec{y}_{M_i})\rightarrow \hat{\psi}_\sigma^i(\vec{y}_{M_i})\big)\Big),\]
where $\hat{\psi}_\sigma^i$ is defined by replacing occurrences of $C_k(-)$ in $\psi_\sigma^i$ with $R_k(\vec{x}_N,-)$ for all $k\in\{1,\ldots,K\}$.

For $\sigma\in\Sigma^{>0}$, define $T_\sigma=\{t_\sigma\}\cup \{t_\sigma^i:i\in I\}$. If the order of $\sigma$ is 0 then $\sigma$ is already an $\sL$-sentence, so we define $T_\sigma =\{\sigma\}$ in this case. Define 
\[T=\bigcup_{\sigma\in \Sigma} T_\sigma.\]
Then $T$ is a theory for the expanded signature $\sL' = \bigcup_{\sigma\in\Sigma} \sL'_\sigma$. We assume that if $\sigma_1\neq\sigma_2$ then the extra symbols added to $\sL'_{\sigma_1}$ and  $\sL'_{\sigma_2}$ are all distinct. Define
\[\cB' = \{A\in \cA : \text{we can interpret the additional symbols of } \sL' \text{ so that } A\models T\}.\]
Let $B\in \cB'$, and let $\sigma\in \Sigma^{>0}$ have order $K$. Then we can interpret the additional symbols of $\sL'$ in $B$ so that $B\models T_\sigma$. In particular, if $\vec{b}_N\in B$ is such that $B\models \mu_\sigma(\vec{b}_N)$, then $B\models \hat{\eta}_\sigma(\vec{b}_N)$, and, assuming that $\tau_\sigma\neq \top$ and given $i\in I$, we also have $B\models \forall \vec{y}_{M_i}\big(\gamma_\sigma^i(\vec{y}_{M_i})\rightarrow \hat{\psi}_\sigma^i(\vec{y}_{M_i})\big)$. So, if $\sL'_\sigma = \sL\cup\{R_1,\ldots,R_K\}$, then for each $k\in \{1,\ldots,K\}$, whenever $B\models \mu_\sigma(\vec{b}_N)$  we can interpret $C_k$ by 
\[C_k(a) \iff R_k(\vec{b}_N,a),\]
and so a routine argument reveals that $B\models \sigma$. There is nothing to do for the case where $\sigma$ has order 0, and so it follows that $\cB'\subseteq \cB$.  

Conversely, if $B\in \cB$ then we can make $B$ into an $\sL'$-structure by interpreting the new relations as follows. If $R$ is one such new relation, then it is associated with a unary predicate symbol $C$ appearing in some separation rule
\[\forall\vec{x}_N\Big(\mu(\vec{x}_N)\rightarrow \exists \vec{C}_K\big(\eta(\vec{x}_N)\wedge \tau \big)\Big).\] 
Let $\vec{b}_N\in B$ and suppose $B\models \mu(\vec{b}_N)$. Then there is an associated instantiation of $C$ in $B$ which we denote $C_{\vec{b}_N}$. Now, define the interpretation of $R$ in $B$ using
\[R = \{(\vec{b}_N,a): B\models \mu(\vec{b}_N)\wedge C_{\vec{b}_N}(a)\}.\]
Then another routine argument reveals that $B\models T$, and thus $\cB\subseteq \cB'$. So $\cB=\cB'$, and $\cB$ is pseudoelementary relative to $\cA$ as required. 

If $\cB$ is essentially finite/countable/r.e. as a separation subclass of $\cA$, then that $\cB$ is essentially finite/countable pseudoelementary/r.e. relative to $\cA$ follows immediately from the construction of $\sL'$ and $T$. 
\end{proof}

\begin{cor}\label{C:pseud}
If $\cB$ is a separation class then $\cB$ is pseudoelementary.
\end{cor}
\begin{proof}
This follows immediately from Lemma \ref{L:pseud} and the definition of separation classes (Definition \ref{D:sep}).
\end{proof}

Converses to Lemma \ref{L:pseud} and Corollary \ref{C:pseud} do not hold in general. To see this, note that we shall show that separation classes are elementary (Theorem \ref{T:same}), while pseudoelementary classes may not be.

The following lemma is a mild generalisation of the well known fact that pseudoelementary classes are closed under ultraproducts.
\begin{lemma}\label{L:ultra}
If $\cA$ is closed under ultraproducts and $\cB$ is pseudoelementary relative to $\cA$, then $\cB$ is closed under ultraproducts.
\end{lemma}
\begin{proof}
Suppose $T$ is an $\sL'$-theory making $\cB$ pseudoelementary relative to $\cA$. Let $I$ be an indexing set and for each $i\in I$ let $B_i\in\cB$. Let $\prod_U B_i$ be an ultraproduct. For every $i$ we can define an $\sL'$-structure on $B_i$, which we denote $B'_i$, such that $B'_i\models T$. Then $\prod_U B'_i\models T$, by \L o\'s's theorem \cite{Los55}, and as $\prod_U B_i \in \cA$ it follows that $\prod_U B_i \in \cB$.
\end{proof}

The aim now is to show that separation subclasses of elementary classes are elementary. In view of Lemmas \ref{L:pseud} and \ref{L:ultra} it will be sufficient to prove they are closed under taking elementary substructures and appeal to \cite[Theorem 2.13]{FMS62}. This is done by the following pair of technical lemmas. 

\begin{lemma}\label{L:subs}
Let $\sL$ be a first-order signature, let $A$ be an $\sL$-structure and let $B$ be an elementary substructure of $A$. Let $a_1,\ldots,a_N$ be elements of $B$, and let $S_1,\ldots, S_K$ be unary predicate symbols not appearing in $\sL$. Define $\sL^+=\sL\cup\{S_1,\ldots,S_K\}$. Let $\eta(\vec{z}_N)$ be a quantifier-free first-order $\sL^+$-formula with free variables from $\vec{z}_N$. For each $k\in\{1,\ldots,K\}$ let $X_k\subseteq A$, and use these sets to make $A$ into an $\sL^+$-structure by interpreting $S_k$ as $X_k$ for each $k$. Similarly, make $B$ into an $\sL^+$-structure by interpreting $S_k$ as $X_k\cap B$ for each $k$. Then $A\models \eta(\vec{a}_N)\iff B\models \eta(\vec{a}_N)$.  
\end{lemma}
\begin{proof}
We proceed by structural induction on $\eta$. Let $\vec{a}_N\in B$. If $\eta$ is a pure $\sL$-formula, i.e. if it involves none of the additional predicates, then that $A\models \eta(\vec{a}_N)$ if and only if $B\models \eta(\vec{a}_N)$ follows immediately from the assumption that $B$ is an elementary substructure of $A$. So the non-trivial base cases are the atomic formulas of form $S_k(t(\vec{a}_N))$ where $t(\vec{x}_N)$ is an $\sL$-term. As $B$ is a substructure of $A$ we have $t(\vec{a}_N)\in B$, and so 
\[A\models S^A_k(t(\vec{a}_N))\iff t(\vec{a}_N)\in X_k\iff t(\vec{a}_N)\in X_k\cap B \iff B\models S^B_k(t(\vec{a}_N)).\] 
The inductive step is routine because $\eta$ is quantifier-free, and so it suffices to deal with $\neg$ and $\wedge$.
\end{proof}

\begin{lemma}\label{L:Csubs}
If $\cB$ is a separation subclass of $\cA$ and $\cA$ is closed under taking elementary substructures, then $\cB$ is also closed under taking elementary substructures.
\end{lemma}
\begin{proof}
Let $\cB$ be a separation subclass of $\cA$, let $\Sigma$ be a separation scheme defining $\cB$ relative to $\cA$, and suppose $\cA$ is closed under taking elementary substructures. Let $B\in\cB$, and let $B'$ be an elementary substructure of $B$. Then $B'\in\cA$, as $\cA$ is closed under taking elementary substructures. We aim to show that $B'\models \Sigma$, and thus that $B'\in\cB$. So, given arbitrary $\sigma\in \Sigma$ we must show $B'\models \sigma$. If the order of $\sigma$ is 0 then this follows immediately from the assumption that $B'$ is an elementary substructure of $B$ (recall Definition \ref{D:Srule}). So suppose 
\[\sigma = \forall\vec{x}_N\Big(\mu(\vec{x}_N)\rightarrow \exists \vec{C}_K\big(\eta(\vec{x}_N)\wedge \bw_I\forall \vec{y}_{M_i}(\gamma_i(\vec{y}_{M_i})\rightarrow \psi_i(\vec{y}_{M_i}))\big)\Big),\]
for some $I\neq \emptyset$, as the case where $\tau = \top$ is similar but easier. 
Let $ a_1,\ldots,a_N\in B'$, and suppose $B'\models \mu(\vec{a}_N)$. Then, as $B'$ is an elementary substructure of $B$ we must have $B\models\mu(\vec{a}_N)$, and thus 
\[B\models \exists \vec{C}_K\big(\eta(\vec{a}_N)\wedge \bw_I\forall \vec{y}_{M_i}(\gamma_i(\vec{y}_{M_i})\rightarrow \psi_i(\vec{y}_{M_i}))\big).\]
 This is equivalent to saying that we can extend $\sL$ with new unary predicate symbols $C_1,\ldots,C_K$ to a signature $\sL^+$, and make $B$ into an $\sL^+$-structure in such a way that $B\models \eta(\vec{a}_N)\wedge \bw_I\forall \vec{y}_{M_i}(\gamma_i(\vec{y}_{M_i})\rightarrow \psi_i(\vec{y}_{M_i}))$ when this is treated as an $\sL^+$-formula in the obvious way. We treat $B'$ as an $\sL^+$ structure by interpreting the new predicates as the restrictions of their interpretations in $B$. We aim to use Lemma \ref{L:subs}.  

First of all, we have $B'\models \eta(\vec{a}_N)$ by immediate application of Lemma \ref{L:subs}. Now, let $i\in I$, let $b_1,\ldots,b_{M_i}\in B'$, and suppose $B'\models \gamma_i(\vec{b}_{M_i})$. Then, as $B'$ is an elementary substructure of $B$, we also have $B\models \gamma_i(\vec{b}_{M_i})$, as $\gamma_i$ is an $\sL$-formula, and thus also $B\models \psi_i(\vec{b}_{M_i})$, as $B\models \sigma$. So, again by Lemma \ref{L:subs}, we have $B'\models \psi_i(\vec{b}_{M_i})$. This is true for arbitrary choices of $b_1,\ldots,b_{M_i}\in B'$, so it follows that $B'\models \forall \vec{y}_{M_i}(\gamma_i(\vec{y}_{M_i})\rightarrow \psi_i(\vec{y}_{M_i}))$ as required. This is true for all $i\in I$, and so $B'\models \sigma$. This is true for all $\sigma\in \Sigma$, and so $B'\models \Sigma$. Since $B'\in\cA$ and $B'\models \Sigma$ it follows that $B'\in \cB$ as required.          
\end{proof}

\begin{thm}\label{T:same}
If $\cA$ is elementary and $\cB$ is a separation subclass of $\cA$ then $\cB$ is also elementary.
\end{thm}
\begin{proof}
$\cB$ is closed under taking ultraproducts, by Lemmas \ref{L:pseud} and \ref{L:ultra},  and is also closed under elementary substructures, by Lemma \ref{L:Csubs}, so the result follows from \cite[Theorem 2.13]{FMS62}. 
\end{proof}

Theorem \ref{T:same} is not constructive, but we will later exploit the fact that we know that separation subclasses of elementary classes \emph{can} be axiomatised to produce explicit axiomatisations.

\section{The separation game}\label{S:game}
We will define games played between two players, Abelard ($\forall$) and Eloise ($\exists$). The point here is that the existence of winning strategies in these games relates to a structure's membership in separation subclasses, in a sense to be made precise in Proposition \ref{P:game} below, and can also be captured in first-order logic, as we explain in Section \ref{S:ax}. So the games provide a means to translate the second-order separation rules defining separation subclasses into first-order logic. 

Our games are played over a fixed $\sL$-structure in rounds numbered by the natural numbers starting with zero. In each round, $\forall$ plays first, then $\exists$ must respond. If a player has no legal move to make when required to play, then that player loses the game immediately, and the game does not continue. If one player loses, then the other player necessarily wins. We say that $\forall$ has an $r$-strategy if he can play in a way that guarantees he wins no later than round $r$. We say $\exists$ has an $r$-strategy if she can play in a way that guarantees that $\forall$ will not win till at least the $(r+1)$th round, either by not losing, or by winning herself prior to that point. We say that $\exists$ has an $\omega$-strategy if she can play in such a way that she can either win or survive indefinitely, however $\forall$ plays.

We now define the rules of our games more precisely. Let $\sL$ be a first-order signature, and let $\cA$ be a class of $\sL$-structures. Let $\Sigma$ be a separation scheme for $\sL$, and let $\sigma=\forall\vec{x}_N\Big(\mu(\vec{x}_N)\rightarrow \exists \vec{C}_K(\eta(\vec{x}_N)\wedge \tau)\Big)\in \Sigma^{>0}$ (recall that $\Sigma^{>0}$ is the subset of $\Sigma$ containing the separation rules of positive order). Let $A\in\cA$, and for each $k\in\{1,\ldots,K\}$ let $S_k,\bar{S}_k\subseteq A$. We define the $(A,\sigma)$\textbf{-game with starting position} $(S_1,\ldots,S_K,\bar{S}_1,\ldots,\bar{S}_K)$. The idea is that, for each $k$, \/ $S_k$ will contain elements definitely specified by the monadic predicate $C_k$, and $\bar{S}_k$ will denote a set of elements that are definitely in its complement. Over the course of the game $\exists$ is forced to decide whether elements of $A$ are, or are not, contained in $S_k$. Note that at any given point in the game, $\bar{S}_k$ will generally be a strict subset of the complement of $S_k$, as there will usually be elements that $\exists$ has not yet been forced to make a decision about. If $\exists$ cannot make a move that does not violate the conditions to be defined below, then she loses the game. Formally, the game is played as follows:

\begin{itemize}
\item In round 0, $\forall$ chooses $a_1,\ldots,a_N\in A$ such that $A\models \mu(\vec{a}_N)$. In response, $\exists$ must decide, for each $n\in\{1,\ldots,N\}$ and $k\in\{1,\ldots,K\}$, whether $a_n\in S_k$. If yes then $a_n$ is added to $S_k$. If no then $a_n$ is added to $\bar{S}_k$. \/ $\exists$ must choose in such a way that:
\begin{enumerate}[(1)]
\item $A\models \eta(\vec{a}_N)$, where $\eta$ is treated as a formula for signature 
\[\sL^+ = \sL\cup\{C_1,\ldots,C_K\},\] 
and $C_k$ is interpreted as $S_k$ for all $k$ (where $S_k$ includes any elements newly added by $\exists$).
\item $S_k\cap \bar{S}_k=\emptyset$ for all $k$.
\end{enumerate}
\item In round $r$ for $r> 0$, $\forall$ must play a move of form $(\tau_i, \vec{b}_{M_i})$, where 
\[\tau_i = \forall \vec{y}_{M_i}\Big(\gamma(\vec{y}_{M_i})\rightarrow \psi(\vec{y}_{M_i})\Big)\] 
is part of the conjunction $\tau=\bw_I \tau_i$, and $b_1,\ldots,b_{M_i}\in A$ such that $A\models \gamma (\vec{b}_{M_i})$. If $\tau = \top$ then $\forall$ cannot do this, and so he loses.

$\exists$ must respond by deciding, for each $m\in \{1,\ldots,M_i\}$ and $k\in\{1,\ldots,K\}$, whether $b_m\in S_k$. If yes then $b_m$ is added to $S_k$, and if no then $b_m$ is added to $\bar{S}_k$. \/ $\exists$ must choose in such a way that: 
\begin{enumerate}[(1)]
\item $A\models \psi(\vec{b}_{M_i})$, where $\psi$ is treated as an $\sL^+$-formula, and $C_k$ is interpreted as $S_k$ for all $k$ (where $S_k$ includes new elements added by $\exists$). 
\item $S_k\cap \bar{S}_k=\emptyset$ for all $k$. 
\end{enumerate}     
\end{itemize}

We sometimes refer to the $(A,\sigma)$-game with starting position $S_k=\bar{S}_k=\emptyset$ for all $k\in \{1,\ldots,K\}$ as the \textbf{simple} $(A,\sigma)$-game. Note that we only define these games for $\sigma$ of positive order, as in the order zero case $\sigma$ is just a first-order sentence, so there is no need `translate' it into first-order logic.  

\begin{prop}\label{P:game}
Let $\cA$ be a class of $\sL$-structures, let $\cB$ be a separation subclass of $\cA$ defined by the separation scheme $\Sigma$, and let $A\in \cA$. Then:
\begin{enumerate}[(1)]
\item If $A\in \cB$, then $\exists$ has an $\omega$-strategy for the simple $(A,\sigma)$-game for all $\sigma\in \Sigma^{>0}$. 
\item If $\cB$ is essentially countable and $A$ is countable, then the converse is true. I.e. if $A\models \sigma$ for all $\sigma\in\Sigma\setminus\Sigma^{>0}$, and if $\exists$ has an $\omega$-strategy for the simple $(A,\sigma)$-game for all $\sigma\in \Sigma^{>0}$, then $A\in\cB$.  
\end{enumerate}
\end{prop}   
\begin{proof}
For 1), if $A\in \cB$ then $A\models \sigma$, so, given $\vec{a}_N\in A$ with $A\models \mu(\vec{a}_N)$, there are monadic predicates $C_1,\ldots,C_K$ such that $A\models \eta(\vec{a}_N)\wedge \tau$. In this case $\exists$ can guarantee to never lose by assigning an element $b\in A$ to $S_k$ if $A\models C_k(b)$, and to $\bar{S}_k$ otherwise, whenever she is forced to make a choice.

For 2), suppose that $A$ is countable, and that $\exists$ has an $\omega$-strategy for the simple $(A,\sigma)$-game for every $\sigma\in\Sigma^{>0}$. Let 
\[\sigma= \forall\vec{x}_N\Big(\mu(\vec{x}_N)\rightarrow \exists \vec{C}_K(\eta(\vec{x}_N)\wedge \tau)\Big)\in\Sigma^{>0},\]
 let $a_1,\ldots,a_N\in A$ with $A\models \mu(\vec{a}_N)$, and suppose $A\not\models \mu(\vec{a}_N)\rightarrow \exists \vec{C}_K(\eta(\vec{a}_N)\wedge \tau)$. Then it follows that $A\not\models \exists \vec{C}_K(\eta(\vec{a}_N)\wedge \tau)$. 

If $\tau = \top$, then $\exists$ has no response to the opening move $\vec{a}_N$ by $\forall$, contradicting the assumption that she has an $\omega$-strategy in the game. So we assume that $\tau = \bw_I\tau_i$ for some non-empty $I$. Since $\Sigma$ is essentially countable and $A$ is countable, we can order the moves $(\tau_i, \vec{b}_M)$ that $\forall$ could potentially make using the natural numbers. Suppose $\forall$ plays according to the strategy whereby in round 0 he plays $\vec{a}_N$, and in every subsequent round he plays the lowest ranked legal move that he has not yet played. 

Consider the sets $S_k$ for $k\in \{1,\ldots,K\}$ constructed by $\exists$ as she follows her $\omega$-strategy against $\forall$. In other words, each $S_k$ is the union of the corresponding sets from the individual rounds of the game.  By the rules governing the opening round of play, and the assumption that $\exists$ is playing according to an $\omega$-strategy, we must have $A\models \eta(\vec{a}_N)$, if $C_k$ is interpreted as $S_k$ for all $k$. Thus, if  $A\not\models \exists \vec{C}_K(\eta(\vec{a}_N)\wedge \tau)$ there must be some 
$i\in I$ with $\tau_i= \forall \vec{y}_{M_i}\Big(\gamma(\vec{y}_{M_i})\rightarrow \psi(\vec{y}_{M_i})\Big)$ such that $A\not\models \tau_i$, where $C_k$ is interpreted as $S_k$ for all $k$. 

It follows that there must be $b_1,\ldots, b_{M_i}\in A$ with $A\models \gamma(\vec{b}_{M_i})$ and $A\not\models \psi(\vec{b}_{M_i})$, again interpreting $C_k$ as $S_k$ for all $k$. But this corresponds to a legal move by $\forall$, so he must have played it at some point, as his strategy implies that he eventually plays every move that becomes available after the first round. Thus we must have $A\models \psi(\vec{b}_{M_i})$ when $C_k$ is interpreted as $S_k$ after all, as $\exists$ is following an $\omega$-strategy. This would be a contradiction. Thus we must have $A\models \exists \vec{C}_K(\eta(\vec{a}_N)\wedge \tau)$. Since this is true for every choice of $\vec{a}_N$ such that $A\models \mu(\vec{a}_N)$, we have $A\models \sigma$, and since this argument holds for all $\sigma\in \Sigma^{>0}$, and we have assumed that $A\models \sigma$ for all $\sigma\in\Sigma\setminus\Sigma^{>0}$, it follows that $A\in \cB$ as required.        
\end{proof} 

Note that round 0 is conceptually distinct from the subsequent rounds. We define the \textbf{reduced} $(A,\sigma)$\textbf{-game with starting position} $(S_1,\ldots,S_K,\bar{S}_1,\ldots,\bar{S}_K)$ to be the $(A,\sigma)$-game with the same starting position, but omitting round 0. For convenience we keep the same labeling for rounds as in the normal game, so the reduced game starts with round 1, not round 0. The concept of an $r$-strategy for $r\geq 1$ carries over without modification for both players.

\section{Generating recursive axiomatisations}\label{S:ax}
Given a separation scheme $\Sigma$, the next step is to find a set of first-order axioms equivalent to $\exists$ having an $\omega$-strategy in the simple $(A,\sigma)$-game for all $\sigma\in\Sigma^{>0}$. We must assume that $\Sigma$ is at least essentially recursively enumerable for the main result (Theorem \ref{T:main}) to hold. We also assume for convenience that for every $\sigma\in\Sigma^{>0}$ the associated $\tau$ is of form $\bw_{i\in\omega} \tau_i$. Recall that in practice, $\tau$ may be $\top$, or a conjunction of only finitely many formulas. We could avoid this assumption by dividing several of the definitions and proofs to come into `finite' and `infinite' cases, but we trust instead that the necessary alterations for the finite cases will become clear once the infinite case is understood.  

Writing down these axioms will involve some quite intricate notational constructions, and we will benefit greatly later from taking the time now to prove some technical results. Recall that for distinct variable symbols $x_1,\ldots,x_N$ we use the notation $\vec{x}_N=\{x_1,\ldots,x_N\}$. If $v$ is a valuation in the model theoretic sense, we will often write e.g. $v[\vec{x}_N]$ to stand for $\{v(x_1),\ldots,v(x_N)\}$. Similarly, if $Z$ is a set of variables we will use $v[Z]$ to denote $\{v(z):z\in Z\}$.

First we want to formalise in first-order logic the statement that $\exists$ can survive a certain finite number of rounds in an $(A,\sigma)$-game from a given starting position. We will do this by recursion. More explicitly, the statement that $\exists$ has an $(r+1)$-strategy will be formed by writing a formula to the effect that, whatever move $\forall$ makes, $\exists$ can distribute the elements picked out by $\forall$ in such a way that she has an $r$-strategy in the $(A,\sigma)$-game with the starting position that results from her choice.

To make this work, the states of the predicates $C_k$ for $k\in\{1,\ldots,K\}$ associated with the separation rule $\sigma$ at a given point in the game will be captured using free variables introduced for this purpose and an assignment $v$ of variables to elements of $A$. We will do this with free variables grouped into sets $Z_1,\ldots,Z_{K}$. The idea is that for $k\in\{1,\ldots,K\}$, the set $v[Z_k]$ captures the elements $\exists$ has assigned to $C_k$. Lemma \ref{L:pad} below explains how we can translate formulas involving the predicates $C_1,\ldots,C_K$ into formulas where they are captured by free variables in this way. 

\begin{lemma}\label{L:pad}
Let $\sL^+=\sL\cup\{C_1,\ldots,C_K\}$, where each $C_k$ is a unary predicate symbol not appearing in $\sL$, let $Y$ be a finite set of variable symbols, and let $\psi(Y)$ be a quantifier-free $\sL^+$-formula whose free variables are taken from $Y$. For each $k\in \{1,\ldots,K\}$ let $Z_k$ be a finite set of variable symbols. Let $\vec{Z}_K=(Z_1,\ldots,Z_K)$. Then we can define a quantifier-free $\sL$-formula $\psi_{\vec{Z}_K}$
whose free variables are from the set $Y\cup \bigcup _{k=1}^{K} Z_k$, such that, whenever $A$ is an $\sL$-structure and $v$ is an assignment, we have
\[A,v \models_{\sL} \psi_{\vec{Z}_K}(Y\cup \bigcup _{k=1}^{K} Z_k) \iff A, v \models_{\sL^+} \psi(Y),\]
where $A$ is treated as an $\sL^+$-structure by interpreting $C_k$ so that 
\[A,v\models C_k(x)\iff v(x)\in v[Z_k],\] for all $k\in\{1,\ldots,K\}$.
\end{lemma}
\begin{proof}
We use induction on the construction of $\psi$. If $\psi$ is an atomic formula, then either:
\begin{enumerate}[(1)]
\item $\psi(Y) = R(t_1(Y),\ldots, t_N(Y))$, where $R$ is some $N$-ary relation symbol from $\sL$ and $t_n$ is an $\sL$-term for each $n\in \{1.\ldots,N\}$,
\item $\psi(Y) = t_1(Y) \approx t_2(Y)$ where $t_1$ and $t_2$ are $\sL$-terms, or  
\item $\psi(Y) = C_k(t(Y))$, where $t$ is an $\sL$-term and $k\in \{1,\ldots,K\}$.
\end{enumerate}
In cases (1) and (2), the interpretation of the additional predicates of $\sL^+$ isn't relevant, so we can define 
\[\psi_{\vec{Z}_K }(Y)=\psi(Y).\]
In the third case, define 
\[\psi_{\vec{Z}_K}(Y\cup Z_k)= (\bv_{z\in Z_k}t(Y)\approx z).\]
Then 
\begin{align*}
&\phantom{\iff i}A,v \models_{\sL} \psi_{\vec{Z}_K}(Y\cup Z_k )\\
&\iff v(t(Y)) = v(z)\text{ for some }z\in Z_k\\ 
&\iff A, v \models_{\sL^+} \psi(Y)\text{, where $C_k$ is interpreted as $v[Z_k]$ as described}.
\end{align*}  
For the inductive step, consider first $\neg\psi$ such that $\psi_{\vec{Z}_K }$ is known to exist for $\psi$. In this case we can just use $\neg\psi_{\vec{Z}_K}$, as

\begin{align*}
A, v \models_\sL \neg\psi_{\vec{Z}_K} &\iff A, v \not\models_{\sL} \psi_{\vec{Z}_K} \\
&\iff A, v \not\models_{\sL^+} \psi \\
&\iff A, v\models_{\sL^+} \neg\psi.
\end{align*}

Consider next $\psi^1\vee \psi^2$, such that appropriate $\psi^1_{\vec{Z}_K}$ and $\psi^2_{\vec{Z}_K}$ exist. We use $\psi^1_{\vec{Z}_K}\vee\psi^2_{\vec{Z}_K}$, because

\begin{align*}
A, v \models_\sL \psi^1_{\vec{Z}_K}\vee \psi^2_{\vec{Z}_K} &\iff A, v \models_{\sL} \psi^1_{\vec{Z}_K } \text{ or } A, v \models_{\sL} \psi^2_{\vec{Z}_K} \\
&\iff A, v \models_{\sL^+} \psi^1 \text{ or } A, v \models_{\sL^+} \psi^2 \\
&\iff A, v\models_{\sL^+} \psi^1\vee \psi^2.
\end{align*}

Since $\psi$ is quantifier-free, we are done.
\end{proof}

As explained above, the state of each $C_k$ during the $(A,\sigma)$-game will be captured by $v[Z_k]$ for some set $Z_k$ of variable symbols and assignment $v$. During the game, $\exists$ must assign elements both to $C_k$ and its complement. For each $k$, we will use a set $\bar{Z}_{k}$ of variable symbols for this purpose. Explicitly, $v[\bar{Z}_{k}]$ captures the elements that have been assigned to the complement of $C_k$. In each round of the game, $\exists$ is presented with new elements of $A$ that she must assign to $C_k$ sets or their complements. These new elements will be captured using $v$ and a new set of variable symbols, e.g. as $v[Y]$.

As discussed previously, a key idea for us is to define for each $r\geq 1$ formulas to the effect of `given a starting position $(v[Z_1],\ldots,v[Z_K],v[\bar{Z}_{1}],\ldots,v[\bar{Z}_{K}])$ in an $(A,\sigma)$-game, whatever move $\forall$ makes involving elements $v[Y]$, \/ $\exists$ can for each $k$ assign each element of $v[Y]$ to either $v[Z_k]$ or $v[\bar{Z}_{k}]$ to obtain $v[Z'_k]$ and $v[\bar{Z}'_{k}]$ in such a way that she has an $(r-1)$-strategy in the $(A,\sigma)$-game with starting position $(v[Z'_1],\ldots,v[Z'_K],v[\bar{Z}'_{1}],\ldots,v[\bar{Z}'_{K}])$'. As mentioned above, this will be done using recursion, so the formula stating that $\exists$ has an $(r-1)$-strategy in the $(A,\sigma)$-game with starting position $(v[Z'_1],\ldots,v[Z'_K],v[\bar{Z}'_{1}],\ldots,v[\bar{Z}'_{K}])$ must involve the variables of $Y$, but now cast in specified roles as members of the sets $Z'_1,\ldots,Z'_{K},\bar{Z}'_{1},\ldots,\bar{Z}'_{K}$ (as these sets of variables are used to track the members of $C_1,\ldots,C_K$).   

The next definition sets up a notation for this process of adding new variables to sets. The situation to be described is that we have sets of variables $Z_1,\ldots,Z_{K},\bar{Z}_{1},\ldots,\bar{Z}_{k}$, and another set of variables $Y$. For each $y\in Y$ and for each $1\leq k \leq K$ we want to add $y$ to either $Z_k$ or to $\bar{Z}_{k}$. What we do for a given $y\in Y$ and $k\in\{1,\ldots,K\}$ is controlled by a function $f:Y\times \{1,\ldots,K\}\to \{0,1\}$. Explicitly, if $f(y,k)= 1$ then we add $y$ to $Z_k$, otherwise we add $y$ to $\bar{Z}_{k}$. The functions $\Delta_K$ and $\Delta_{\bar{K}}$ in Definition \ref{D:delta} below formalise this. The input is the data of the sets $Z_1,\ldots,Z_{K},Y$ (for $\Delta_K$), or $Z_1,\ldots,Z_{K},\bar{Z}_1,\ldots,\bar{Z}_{K},Y$ (for $\Delta_K^+$),  and the `control' function $f$, and the output is either $Z_1,\ldots,Z_{K}$ (for $\Delta_K$),  or  $Z_1,\ldots,Z_{K},\bar{Z}_1,\ldots,\bar{Z}_{K},Y$ (for $\Delta_K^+$) after the elements of $Y$ have been added as just described. We define two functions because sometimes we care about the $C_k$ sets and their complements, and other times just the sets themselves.  

\begin{defn}[$\Delta_K$, $\Delta_{\bar{K}}$, $F^K_Y$]\label{D:delta}
Given a set $Y$ and $1\leq K<\omega$, let $f:Y\times \{1,\ldots,K\}\to\{0,1\}$. For each $k\in \{1,\ldots,K\}$ let $Z_k$ and $\bar{Z}_k$ be sets. We use the shorthand $\vec{Z}_K=(Z_1,\ldots,Z_K)$, and $\vec{\bar{Z}}_K=(\bar{Z}_{1},\ldots,\bar{Z}_{K})$. Define
\[\Delta_K(\vec{Z}_K, f) = (Z'_1,\ldots, Z'_{K}),\]
where for $k\in\{1,\ldots,K\}$ we have
\[Z'_k = Z_k\cup \{y\in Y : f(y,k)=1\}.\] 

Similarly, define 
\[\Delta_{\bar{K}}(\vec{Z}_K, \vec{\bar{Z}}_K, f) = (Z'_1,\ldots, Z'_{K},\bar{Z}'_1,\ldots,\bar{Z}'_K),\]
where $Z'_k$ is as above for all $k\in \{1,\ldots,K\}$, and 
\[\bar{Z}'_k = \bar{Z}_k\cup \{y\in Y : f(y,k)=0\}.\]

We will use $F^K_Y$ to denote the set of functions from $Y\times \{1,\ldots,K\}$ to $\{0,1\}$.
\end{defn}

We now define $\sL$-formulas as follows, noting the assumptions made about $\Sigma$ stated at the start of this section. We assume we are working with a countably infinite pool of variable symbols.  

\begin{enumerate}[$\bullet$]
\item For each $1\leq K< \omega$, and for each $\vec{Z}_K=(Z_1,\ldots,Z_{K})$ and $\vec{\bar{Z}}_K=(\bar{Z}_1,\ldots,\bar{Z}_{K})$ such that $Z_k$ and $\bar{Z}_k$ are finite sets of variables for all $k\in \{1,\ldots,K\}$, define 
\[\mathsf{D}_{(\vec{Z}_K,\vec{\bar{Z}}_K)} \]
to be a quantifier-free $\sL$-formula with free variables exactly $\bigcup_{k=1}^{K} (Z_k \cup \bar{Z}_k) $ such that 
\[\phantom{\iff i}A, v\models \mathsf{D}_{(\vec{Z}_K,\vec{\bar{Z}}_K)} \iff v[Z_k]\cap v[\bar{Z}_{k}]=\emptyset\text{ for all }k\in \{1,\ldots,K\}.\]
\item For each $1\leq K< \omega$, for each $i\in \omega$, for each $\sigma\in\Sigma^{>0}$, and for each $(\vec{Z}_K,\vec{\bar{Z}}_K)$, define 
\[\alpha^\sigma_{(\vec{Z}_K,\vec{\bar{Z}}_K) 0i} = \mathsf{D}_{(\vec{Z}_K,\vec{\bar{Z}}_K) } .\]
\item For each 
\[\sigma=\forall\vec{x}_N\Big(\mu(\vec{x}_N)\rightarrow \exists \vec{C}_K\big(\eta(\vec{x}_N)\wedge \bw_{i\in\omega}\forall \vec{y}_{M_i}(\gamma^i(\vec{y}_{M_i})\rightarrow \psi^i(\vec{y}_{M_i}))\big)\Big)\in \Sigma^{>0},\] for each $i\in\omega$, for each $1\leq r < \omega$, and for each $(\vec{Z}_K,\vec{\bar{Z}}_K)$, recursively define
\[\alpha^\sigma_{(\vec{Z}_K,\vec{\bar{Z}}_K) ri}=\bw_{j\leq i}\forall \vec{y}_{M_j}\Big(\gamma^j(\vec{y}_{M_j})\ra \bv_{f\in F^K_{\vec{y}_{M_j}}} \big(\psi^j_{\Delta_K(\vec{Z}_K, f)}\wedge \alpha^\sigma_{\Delta_{\bar{K}}(\vec{Z}_K,\vec{\bar{Z}}_K ,f)(r-1)i} \big)\Big),\]
where $\psi^j_{\Delta_K(\vec{Z}_K, f)}$ is constructed from $\psi^j$ as in Lemma \ref{L:pad}. In other words,\[A,v\models \psi^j_{\Delta_K(\vec{Z}_K, f)}(\vec{y}_{M_j}\cup\bigcup_{k=1}^K Z'_k)\iff A,v \models_{\sL^+} \psi^j(\vec{y}_{M_j}),\] 
where each $C_k$ is interpreted as $v[Z'_k]$, and $Z'_k$ is constructed from $Z_k$ and $\vec{y}_{M_j}$ according to $f$ (recall Definition \ref{D:delta}).  What these formulas are intended to capture is the idea that $\exists$ can respond to all moves involving $\tau_j$ for $j\leq i$ played by $\forall$, and moreover can do so in such a way that she will continue to be able to respond successfully for at least $r$ rounds. This will be made precise in Lemma \ref{L:reduced}. Note that although it is not apparent from the notation, we are assuming that every new occurrence of $\vec{y}_{M_j}$ in the construction of these formulas involves only fresh variable symbols. If we allow variable symbols to be repeated then it turns out we do not properly capture the concept of `adding elements to $Z_k$', which is what the $\Delta$ operations are supposed to be for. This is explained in the proof of the following lemma.
\end{enumerate}

\begin{lemma}\label{L:reduced}
Let $1\leq K<\omega$, let $A\in\cA$, let $\cB$ be a separation subclass of $\cA$ defined by the essentially r.e. separation scheme $\Sigma$, let $1\leq K,<\omega$, and let $\sigma\in \Sigma$ be a separation rule of order $K$. Then for all finite $Z_1,\ldots,Z_{K},\bar{Z}_1,\ldots,\bar{Z}_K$, for all assignments $v$, for all $i\in I$ and for all $1\leq r< \omega$, the following are equivalent:
\begin{enumerate}[(1)]
\item $A,v\models \alpha^\sigma_{(\vec{Z}_K,\vec{\bar{Z}}_K) ri}$.
\item $\exists$ has an $r$-strategy in the reduced $(A,\sigma)$-game with starting position 
\[(v[Z_1],\ldots,v[Z_{K}],v[\bar{Z}_1],\ldots,v[\bar{Z}_{K}])\] 
where $\forall$ can only play moves involving $\tau_j$ when  $j\leq i$.
\end{enumerate}
\end{lemma}
\begin{proof}
We use induction on $r$. For the base case ($r = 1$), let $i\in\omega$ and suppose first that $A,v\models \alpha^\sigma_{(\vec{Z}_K,\vec{\bar{Z}}_K) 1i}$. Let $v'$ be an assignment agreeing with $v$ about everything except, possibly, $\vec{y}_{M_j}$, and suppose $A,v'\models \gamma^j(\vec{y}_{M_j})$ for some $j\leq i$. Then there is $f\in F^K_{\vec{y}_{M_j}}$ such that $A,v'\models\psi^j_{\vec{Z}'_K} \wedge \mathsf{D}_{(\vec{Z}'_K,\vec{\bar{Z}}'_K)}$ where  $(Z'_1,\ldots, Z'_{K})= \Delta_K(Z_1,\ldots,Z_{K},f)$ and $(Z'_1,\ldots, Z'_{K},\bar{Z}'_1,\ldots, \bar{Z}'_{K})= \Delta_{\bar{K}}(\vec{Z}_K,\vec{\bar{Z}}_K,f)$. I.e. $Z'_k$ is the new value of $Z_k$ as controlled by $f$, and similar for $\bar{Z}'_k$, for each $k\in\{1,\ldots,K\}$ (recall Definition \ref{D:delta}). Because the variables of $\vec{y}_{M_j}$ do not appear in any $Z_k$ (by the assumption mentioned after the definition of $\alpha^\sigma_{(\vec{Z}_K,\vec{\bar{Z}}_K) ri}$), we have $v[Z_k]=v'[Z_k]$ for all $k\in\{1,\ldots,K\}$, and similar for $\bar{Z}_k$. 

So, suppose $\forall$ chooses $\vec{a}_{M_j}\in A$ and plays the move $(\tau_j,\vec{a}_{M_j})$. We define $v'$ to be $v$ except that $v'(y_m)= a_m$ for all $m\in\{1,\ldots,M_j\}$. As discussed above, there is $f$ such that $A,v'\models\psi^j_{\Delta_K(\vec{Z}_K,f)} \wedge \mathsf{D}_{\Delta_{\bar{K}}(\vec{Z}_K,\vec{\bar{Z}}_K,f)}$. So, appealing to Lemma \ref{L:pad} applied to $\psi^j_{\Delta_K(\vec{Z}_K,f)}$, there is a way, described by $f$, to assign each $a_m$ to either $v[Z_k]$ or $v[\bar{Z}_{k}]$ for each $k\in\{1,\ldots,K\}$ such that $A\models_{\sL^+} \psi^j$ if we interpret each $C_k$ as the modified $v[Z_k]$. Moreover, from the definition of $\mathsf{D}$, we see that each modified $v[Z_k]$ is disjoint from the modified $v[\bar{Z}_{k}]$, and so $\exists$ can survive the first round of the reduced $(A,\sigma)$-game with starting position $(v[Z_1],\ldots,v[Z_{K}],v[\bar{Z}_1],\ldots,v[\bar{Z}_{K}])$, so long as $\forall$ starts with a move involving $\tau_j$ for some $j\leq i$. This proves that (1)$\implies$(2) for $r=1$. 

For the converse, suppose (2) holds and that $A,v'\models \gamma^j(\vec{y}_{M_j})$ for some $j\leq i$ with $v'$ agreeing with $v$ about everything except, possibly, $\vec{y}_{M_j}$. Then $\exists$'s strategy tells us how to find $f\in F^K_{\vec{y}_{M_j}}$ appropriately. I.e. if $\exists$'s strategy assigns $v'(y_m)$ to $C_k$, then $f$ should assign $y_m$ to $Z_k$ (formally, $f(y_m,k)=1$), and if not it should assign $y_m$ to $\bar{Z}_{k}$ (formally, $f(y_m,k)=0$). With this $f$, if $Z'_k$ and $\bar{Z}'_k$ are, respectively, the modified $Z_k$ and $\bar{Z}_k$ for all $k\in\{1,\ldots,K\}$, we have $v'[Z'_k]\cap v'[\bar{Z}'_{k}]=\emptyset$, and $A,v'\models_{\sL^+} \psi^j(\vec{y}_{M_j})$ if we interpret $C_k$ as $v'[Z'_k]$ for all $k\in \{1,\ldots,K\}$. From Lemma \ref{L:pad} it follows that $A,v'\models \psi^j_{\Delta_K(\vec{Z}_K, f)}\wedge \mathsf{D}_{\Delta_{\bar{K}}(\vec{Z}_K,\vec{\bar{Z}}_K ,f)}$. So (2)$\implies$(1). 

For the inductive step, let $1< R<\omega$ and suppose the claim is true for all $1\leq r < R$. Then, by the inductive hypothesis, and appealing to similar reasoning as used for the base case, $A, v\models \alpha^\sigma_{(\vec{Z}_K,\vec{\bar{Z}}_K) Ri}$ if and only if, whatever move involving $\tau_j$ for $j\leq i$ $\forall$ plays, $\exists$ can respond in such a way that she has an $(R-1)$-strategy in the game whose starting position corresponds to her response. But this is the same as saying that $\exists$ has an $R$-strategy as claimed.
\end{proof}

The formulas $\alpha^\sigma_{(\vec{Z}_K,\vec{\bar{Z}}_K) ri}$ we have just defined let us deal with reduced $(A,\sigma)$-games, but we need more to handle the full games. With that in mind we now define formulas where round zero is included. These will include the $\alpha$ formulas as subformulas.

\begin{enumerate}[$\bullet$]
\item Let $\sigma\in \Sigma^{>0}$, and for all $i\in \omega$, for all finite $(\vec{Z}_K,\vec{\bar{Z}}_K)$, and for all $0\leq r<\omega$,  define
\[\beta^\sigma_{(\vec{Z}_K,\vec{\bar{Z}}_K)ri} = \forall \vec{x}_N\Big(\mu(\vec{x}_N)\ra  \bv_{f\in F^K_{\vec{x}_N}} \big(\eta_{\Delta_K(\vec{Z}_K, f)}\wedge \alpha^\sigma_{\Delta_{\bar{K}}(\vec{Z}_K,\vec{\bar{Z}}_K ,f)ri}\big)\Big).\]
Here $\eta_{\Delta_K(\vec{Z}_K, f)}$ is constructed from $\eta$ as in Lemma \ref{L:pad}. We assume that the variables in $\vec{x}_N$ do not appear in any subformula not involving $\mu$. The formula $\beta^\sigma_{(\vec{Z}_K,\vec{\bar{Z}}_K)ri}$ is intended to say that, whatever move $\forall$ makes in round zero of the $(A,\sigma)$-game with starting position $(\vec{Z}_K,\vec{\bar{Z}}_K)$, \/ $\exists$ can respond in such a way that she survives this opening round, and can survive in the resulting reduced game for at least $r$ rounds, so long as $\forall$ only plays moves involving $\tau_j$ for $j\leq i$. Lemma \ref{L:full} makes this precise.
\end{enumerate}

\begin{lemma}\label{L:full}
Let $1\leq K<\omega$, let $A\in\cA$, let $\cB$ be a separation subclass of $\cA$ defined by the essentially r.e. separation scheme $\Sigma$, and let $\sigma\in \Sigma$ be a separation rule of order $K$. Then for all $(\vec{Z}_K,\vec{\bar{Z}}_K)$, for all assignments $v$, for all $i\in I$, and for all $r\in \omega$, the following are equivalent:
\[\tag{1}A,v\models \beta^\sigma_{(\vec{Z}_K,\vec{\bar{Z}}_K)ri}.\]
\begin{align*}\tag{2}&\exists \text{ has an  $r$-strategy in the $(A,\sigma)$-game with starting position} \\ 
&\hspace{1cm}(v[Z_1],\ldots,v[Z_{K}],v[\bar{Z}_1],\ldots,v[\bar{Z}_K])\\ 
&\text{ where $\forall$ can only play moves involving $\tau_j$ when  $j\leq i$.} 
\end{align*} 
\end{lemma}
\begin{proof}
(1) is the statement that whenever $v'$ is an assignment agreeing with $v$ about everything except, possibly, $\vec{x}_N$, if $A,v'\models \mu(\vec{x})$ then there is a way $\exists$ can assign the variables of $\vec{x}_N$ 
to $\vec{Z}_K,\vec{\bar{Z}}_K$ so that $A,v'\models \eta_{\vec{Z}'_K}$ and $A,v'\models \alpha^\sigma_{(\vec{Z}'_K,\vec{\bar{Z}}'_K)ri}$ for the resulting values $Z'_1,\ldots,Z'_{K},\bar{Z}'_1,\ldots,\bar{Z}'_{K}$. 

So, suppose $\forall$ plays $\vec{a}_N$ as an opening move. Define $v'$ to be like $v$ except that $v'(x_n) = a_n$ for all $n\in\{1,...,N\}$. By the assumption that $\vec{x}_N$ involves only variables not appearing in subformulas not involving $\mu$, we have $v[Z_k]=v'[Z_k]$ for all $k$. Then $A,v'\models \mu(\vec{x}_N)$, and by the preceding paragraph, there is a way $\exists$ can assign variables from $\vec{x}_N$ so that $A,v'\models \eta_{\vec{Z}'_K}$ and $A,v'\models \alpha^\sigma_{(\vec{Z}'_K,\vec{\bar{Z}}'_K)ri}$. So, for her response, $\exists$ adds $a_n$ to $v[Z_k]$ if $x_n$ is added to $Z_k$, and otherwise adds it to $v[\bar{Z}_k]$. Since $A,v'\models \eta_{\vec{Z}'_K}$, it follows from Lemma \ref{L:pad} that $\exists$ can survive the opening round, and since $A,v'\models \alpha^\sigma_{(\vec{Z}'_K,\vec{\bar{Z}}'_K)ri}$, it follows from Lemma \ref{L:reduced} that she has an $r$-strategy in the resulting reduced game. Thus (1)$\implies$(2).

Conversely, if $\exists$ has an $r$-strategy in the game, if $v'$ is an assignment agreeing with $v$ about everything except possibly $\vec{x}_N$ and $A,v'\models \mu(\vec{x}_N)$, then this strategy implies the existence of a suitable $f$. Explicitly, $f(x_n,k)=1$ if $v'(x_n)$ is added to $v[Z_k]$, and $f(x_n,k)=0$ otherwise. Thus (2)$\implies$(1) as required.
\end{proof}

The formulas $\beta^\sigma_{(\vec{Z}_K,\vec{\bar{Z}}_K)ri}$ defined above say that $\exists$ has an $r$ strategy in a constrained $(A,\sigma)$-game with a specified starting position. What we want for our main result are formulas stating that this is true for the `simple' starting position. I.e. where the sets all start off empty. We also want to cover separation rules of order zero. With is in mind, we proceed as follows. 
\begin{enumerate}[$\bullet$]
\item For all $\sigma\in \Sigma$, for all $i\in \omega$ and for all $r\in\omega$,  define
\[\hat{\beta}^\sigma_{ri} = \begin{cases} \beta^{\sigma}_{(\vec{\emptyset}_K,\vec{\emptyset}_K)ri} \text{ if } \sigma\in\Sigma^{>0} \text {has order $K$},\\
\sigma \text{ otherwise}.
\end{cases}\]
\end{enumerate}
Here $\vec{\emptyset}_K$ stands for $\vec{Z}_K$ where $Z_k=\emptyset$ for all $k$. So, given any $\sigma\in\Sigma^{>0}$, and given any $r,i\in\omega$, the first-order sentence $\hat{\beta}^\sigma_{ri}$ states that $\exists$ has an $r$-strategy in the simple $(A,\sigma)$-game in which $\forall$ can only play moves involving $\tau_j$ for $j\leq i$. Moreover, if $\Sigma$ is essentially r.e., the set $\{\hat{\beta}^\sigma_{ri}: \sigma\in \Sigma, i,r\in\omega\}$ can be made recursively enumerable.   

This brings us to the following theorem, which is the main result of this section, and is a considerable generalisation of \cite[Theorem 4.5]{HirMcL17} and \cite[Theorem 5.6]{Egr19}. Nevertheless, the key ingredients of the proofs are essentially the same.

\begin{thm}\label{T:main}
Let $\cA$ be an elementary class of $\sL$-structures, let $A\in\cA$, and let $\cB$ be a separation subclass of $\cA$ defined by the essentially r.e. separation scheme $\Sigma$.  Then 
\[A\in \cB\iff A\models \hat{\beta}^\sigma_{ri} \text{ for all }\sigma\in \Sigma \text{ and for all } r,i\in \omega.\]
\end{thm}
\begin{proof}
If $A\in \cB$ then, for all $\sigma=\forall\vec{x}_N\Big(\mu(\vec{x}_N)\rightarrow \exists \vec{C}_K(\eta(\vec{x})\wedge \tau)\Big)\in \Sigma^{>0}$, we can use the predicates $C_k$ for $k\in \{1,\ldots, K\}$ to guide the strategy of $\exists$ in the appropriate games. By Proposition \ref{P:game}, $\exists$ has an $\omega$-strategy in every simple $(A,\sigma)$-game for $\sigma\in\Sigma^{>0}$, and thus $A\models \beta^{\sigma}_{(\vec{\emptyset}_K,\vec{\emptyset}_K)ri}$ for all $\sigma\in\Sigma^{>0}$ and for all $r,i\in\omega$, by Lemma \ref{L:full}. If $\sigma\in\Sigma\setminus\Sigma^{>0}$ then $\sigma= \hat{\beta}^\sigma_{ri}$, and so it follows immediately that $A\models \hat{\beta}^\sigma_{ri}$ for all $\sigma,r,i$.

Conversely, suppose first that $A$ is countable and that $A\notin \cB$. Then, either $A\not\models \sigma$ for some $\sigma\in\Sigma\setminus \Sigma^{>0}$, or, by Proposition \ref{P:game}, there is $\sigma\in \Sigma^{>0}$ such that $\exists$ does not have an $\omega$-strategy in the simple $(A,\sigma)$-game. In the former case we immediately have $A\not\models \hat{\beta}^\sigma_{ri}$, just by definition of $\hat{\beta}^\sigma_{ri}$, so we consider the latter.  It follows from K\"onig's Tree Lemma \cite{Kon26} that some game tree for the simple $(A,\sigma)$-game is finite (otherwise $\exists$ would have a strategy defined using an infinite branch). There are only a finite number of $\forall$ moves in this game tree, and so, if $i\in \omega$ is the largest index of a $\tau_i$ used in a move by $\forall$ in this tree, we have $A\not\models \hat{\beta}^\sigma_{ri}$ for some $r\in\omega$, by Lemma \ref{L:full}.

Now, suppose $A$ is uncountable, and suppose also that $A\models \hat{\beta}^\sigma_{ri} $ for all $\sigma\in \Sigma$ and for all $r,i\in \omega$. Then, by the downward L\"owenheim-Skolem Theorem, $A$ has a countable elementary substructure, $A'$, and, as $A'\models \hat{\beta}^\sigma_{ri} $ for all $\sigma,i,r$, it follows from our proof of the countable case that $A'\in \cB$. Moreover, by Theorem \ref{T:same}, $\cB$ is elementary, so $A'$ is a model of the elementary theory defining $\cB$. But $A$ and $A'$ are elementarily equivalent, so $A$ is also a model of this theory, and thus $A\in \cB$ as claimed.  
\end{proof} 

\begin{cor}\label{C:main}
Let $\cA$ be an elementary class of $\sL$-structures, let $A\in\cA$, and let $\cB$ be an essentially r.e. separation subclass of $\cA$ defined by the separation scheme $\Sigma$.  Then $\cB$ has a recursive axiomatisation relative to $\cA$.

Moreover, if for every $\sigma = \forall\vec{x}_N(\mu(\vec{x}_N)\rightarrow \exists \vec{C}_K(\eta(\vec{x}_N)\wedge \bw_I\tau_i))\in \Sigma^{>0}$, the prenex normal form of $\mu$ contains no universal quantifiers, and, in addition, for every conjunct $\tau_i=\forall \vec{y}_{M_i}(\gamma_i(\vec{y}_{M_i})\rightarrow \psi_i(\vec{y}_{M_i}))$ of $\tau$,  the prenex normal form of $\gamma_i$ contains no universal quantifiers, then, so long as the prenex normal form of every $\sigma\in\Sigma\setminus\Sigma^{>0}$ is universal, there is a recursive universal axiomatisation of $\cB$ relative to $\cA$. 
\end{cor}
\begin{proof}
Since $\Sigma$ is essentially r.e. the set $\mathcal T=\{\hat{\beta}^\sigma_{ri}:\sigma\in \Sigma,r,i\in \omega\}$ can be assumed to be recursively enumerable. By Theorem \ref{T:main}, we know $\mathcal T$ axiomatises $\cB$ relative to $\cA$, and by Craig's trick, any class with an r.e. axiomatisation relative to a superclass also has a recursive axiomatisation relative to that superclass (see, for example, \cite[Exercise 6.3.1]{Hodg93}).

Now, let $\sigma=\forall\vec{x}_N\Big(\mu(\vec{x}_N)\rightarrow \exists \vec{C}_K(\eta(\vec{x}_N)\wedge \bw_I\tau_i)\Big)\in\Sigma^{>0}$, let $\hat{\beta}^\sigma_{ri}\in\mathcal T$ be one of the generated $\sL$-sentences axiomatising $\cB$ relative to $\cA$. I.e. \[\hat{\beta}^\sigma_{ri}=\forall \vec{x}_N\Big(\mu(\vec{x}_N)\ra  \bv_{f\in F^K_{\vec{x}_N}} \big(\eta_{\Delta_K(\vec{\emptyset}_K, f)}\wedge \alpha^\sigma_{\Delta_{\bar{K}}(\vec{\emptyset}_K,\vec{\emptyset}_K ,f)ri}\big)\Big).\]
 Suppose $\hat{\beta}^\sigma_{ri}$ is not logically equivalent to a universal $\sL$-sentence. Then, in particular the prenex normal form of $\hat{\beta}^\sigma_{ri}$ contains an existential quantifier. Note that $\mu$ is the antecedent of an implication, so if the prenex normal form of $\mu$ contains no universal quantifiers, then, as $\eta$ is quantifier-free, this implies there must be $f\in F^K_{\vec{x}_N}$ such that the prenex normal form of $\alpha^\sigma_{\Delta_{\bar{K}}(\vec{\emptyset}_K,\vec{\emptyset}_K ,f)ri}$ contains an existential quantifier. This requires that $r\geq 1$, as when $r=0$ this formula is just $\mathsf{D}_{\Delta_{\bar{K}}(\vec{\emptyset}_K,\vec{\emptyset}_K,f)}$. 

Now, by definition, for all $\vec{Z}_k$ and $\vec{\bar{Z}}_K$ we have  
\[\alpha^\sigma_{(\vec{Z}_K,\vec{\bar{Z}}_K) ri}=\bw_{j\leq i} \forall \vec{y}_{M_j}\Big(\gamma_j(\vec{y}_{M_j})\ra \bv_{f\in F^K_{\vec{y}_{M_j}}} \big(\psi^j_{\Delta_K(\vec{Z}_K,\vec{\bar{Z}}_K, f)}\wedge \alpha^\sigma_{\Delta_{\bar{K}}(\vec{Z}_K,\vec{\bar{Z}}_K ,f)(r-1)i} \big)\Big).\] 
If for each $i$ the prenex normal form of $\gamma_i$ contains no universal quantifiers, then, as $\psi^i$ is quantifier-free for all $i$, for the prenex normal form of $\alpha^\sigma_{(\vec{Z}_K,\vec{\bar{Z}}_K) ri}$ to contain an existential quantifier it is necessary that the prenex normal form of $\alpha^\sigma_{\Delta_{\bar{K}}(\vec{Z}_K,\vec{\bar{Z}}_K ,f)(r-1)i}$ contains an existential quantifier for some $f\in F^K_{\vec{y}_{M_j}}$. But then the same argument applies to $\alpha^\sigma_{(\vec{Z}_K,\vec{\bar{Z}}_K) (r-1)i}$, and thus we conclude by induction that $\alpha^\sigma_{(\vec{Z}'_K,\vec{\bar{Z}}'_K) 0i}$ contains an existential quantifier for some $\vec{Z}'_K$ and $\vec{\bar{Z}}'_K$, but this is impossible, as it is quantifier-free by definition. 

This proves the claim, because it follows that provided the conditions are met, we can obtain a universal axiomatisation by putting every formula $\hat{\beta}^\sigma_{ri}$ into prenex normal form.    
\end{proof}

The following lemma articulates an essentially trivial but useful observation.

\begin{lemma}\label{L:finite}
Let $\cB$ be a separation subclass of $\cA$ defined using the essentially finite separation scheme $\Sigma$. Suppose there are natural numbers  $r'$ and $i'$ such that, for all $\sigma\in\Sigma^{>0}$ and for all $A\in\cA$, the following statement holds:  
\begin{enumerate}[$\bullet$]
\item \emph{If $\exists$ has an $r'$-strategy in the $(A,\sigma)$-game where $\forall$ uses moves with index at most $i'$, then she has an $\omega$-strategy in the usual $(A,\sigma)$-game.}
\end{enumerate}
Then the axiomatisations produced in the proof of Corollary \ref{C:main} are equivalent to a finite subset of themselves.
\end{lemma} 
\begin{proof}
By Lemma \ref{L:full}, given $\sigma\in\Sigma^{>0}$, \/ $\exists$ having an $r$-strategy in the $(A,\sigma)$-game bounded by $i$ for $A\in\cA$ is equivalent to saying that $A\models \hat{\beta}^\sigma_{ri}$. So if $r'$ and $i'$ exist as claimed we have $\hat{\beta}^\sigma_{r'i'}\models \hat{\beta}^\sigma_{ri}$ for all $r,i$. Thus by Theorem \ref{T:main} we have $A\in\cB\iff A\models \hat{\beta}^\sigma_{r'i'}$ for all $\sigma\in\Sigma$. The result follows as $\Sigma$ is finite.   
\end{proof}

\begin{ex}
Returning to Example \ref{E:poset2}, by Corollary \ref{C:main} we see that the class of representable posets has a recursive axiomatisation (as was proved in \cite{Egr19}). However, the universal quantifiers in the $\join$ and $\meet$ formulas mean that the axiomatisation produced is not universal. Indeed, the class of representable posets has no universal axiomatisation, as it is not closed under substructures (see \cite[Corollary 2.9]{Egr19}).
\end{ex}

We also note the following alternative approach to constructing a recursive axiomatisation for $\cB$ relative to $\cA$ when $\cB$ is an essentially r.e. separation subclass of $\cA$ and $\cA$ is elementary. \cite[Chapter 9]{HirHod02} provides a method for generating a recursive first-order axiomatisation for the elementary closure of any pseudoelementary class whose defining theory in the extended language is recursive. Since in the situation we are describing $\cB$ is elementary (by Theorem \ref{T:same}), the elementary closure is just the class itself, and, since we have an axiomatisation of $\cB$ as an essentially recursive pseudoelementary class by Lemma \ref{L:pseud}, this method can be applied to find a recursive axiomatisation for $\cB$. This method also produces a universal axiomatisation when $\cB$ is pseudouniversal, in the sense of \cite[Definition 9.1]{HirHod02}. 

Also of interest is the result presented as \cite[Theorem 9.14]{HirHod02}, where it is attributed to Mal'cev and Tarski. According to this theorem, every pseudoelementary class that is closed under ultraroots is elementary, and, moreover, if it is also closed under substructures it is universal. If the pseudoelementary theory is r.e. then so too will be the elementary, or universal, axiomatisations. Appropriate sets of axioms are defined, but not made explicit. The reader is directed to the discussion following \cite[Corollary 9.15]{HirHod02} for some comments on this.   

A notable advantage of the recursive axiomatisation generated in the proof of Theorem \ref{T:main} is that, as it has an explicit connection to $\exists$'s ability to survive in certain combinatorial games, it can give us some insight into the question of whether an essentially r.e. separation subclass $\cB$ of an elementary class $\cA$ is finitely axiomatisable relative to $\cA$. To understand how this works, let $T=\{\psi_0, \psi_1, \psi_2,\ldots\}$ be the recursive axiomatisation obtained from Theorem \ref{T:main}. Then, if $\cB$ is finitely axiomatisable relative to $\cA$, there must be some $K\in \omega$ with $A\models \bw_{k=0}^K\psi_k \implies A\models \psi_j$ for all $j\in \omega$, for all $A\in \cA$. So, to prove that no such finite axiomatisation exists, it suffices to construct, for each $K\in\omega$, a structure $A_K\in \cA$ such that $A_K\models \bw_{k=0}^K\psi_k$, but $A_K\not\models \psi_{K+1}$. 

Translating this back into the setting of games, for an essentially finite separation subclass the idea is to construct objects $A_r\in\cA$ such that $\exists$ has $r$-strategies for all simple $(A,\sigma)$-games, but not an $(r+1)$-strategy for at least one such game. The non-finite case is similar, but we must consider the maximum indices of allowed $\forall$ moves, and we also have to take the possibly infinite number of separation rules of order zero into account. Of course, the substance of any such proof is to be found in the constructions themselves, but this can be a useful approach, where it applies. For example, this method is essentially the engine of the proofs of the titular result of \cite{Egr18}, and the results of \cite[section 5]{HirMcL17}, though the work in these examples is phrased in terms of ultraproducts. Note that the argument as described here has an advantage over the originals as reasoning about properties of the ultraproduct is not required.  We present a simple application of this technique in Section \ref{S:N-colour}. 

\section{Expressive power and decision problems}\label{S:decide} 
To begin this section we organise our results on the expressive power of the formalism of separation subclasses vis-\`a-vis first-order logic.

\begin{lemma}\label{L:notFin}
There is a basic elementary class $\cA$, and an essentially finite separation subclass $\cB$ of $\cA$, such that $\cB$ is not finitely axiomatisable.
\end{lemma}
\begin{proof}
We have shown that the class of $(\alpha,\beta)$-representable posets is an essentially finite separation subclass of the class of posets whenever $2\leq\alpha,\beta<\omega$ (see Example \ref{E:poset2}), and this class is also known to not be finitely axiomatisable for $\alpha,\beta\geq 3$ \cite{Egr18}.
\end{proof}

\begin{prop}\label{P:LclassesD}
Let $\cA$ be an elementary class and make the following definitions:
\begin{itemize}
\item[]$\mathbf{S}_\cA$ is the class of separation subclasses of $\cA$.
\item[]$\mathbf{RS}_\cA$ is the class of essentially r.e. separation subclasses of $\cA$.
\item[]$\mathbf{FS}_\cA$ is the class of essentially finite separation subclasses of $\cA$.
\item[]$\mathbf{E}_\cA$ is the class of elementary subclasses of $\cA$.
\item[]$\mathbf{RE}_\cA$ is the class of subclasses of $\cA$ with recursive axiomatisations relative to $\cA$.
\item[]$\mathbf{FE}_\cA$ is the class of subclasses of $\cA$ that are finitely axiomatisable relative to $\cA$.
\end{itemize}
Then Figure \ref{F:LincD} represents the class inclusions that always hold (with arrows from subclass to superclass). In cases where there is no arrow there are choices of $\cA$ for which the inclusion does not hold.
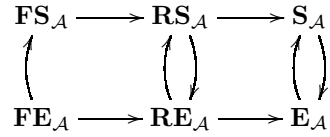
\begin{figure}[htbp]
\centerline{
\xymatrix{\mathbf{FS}_\cA\ar@{->}[r] & \mathbf{RS}_\cA \ar@{->}[r]\ar@/^/@{->}[d] & \mathbf{S}_\cA\ar@/^/@{->}[d] \\
\mathbf{FE}_\cA\ar@{->}[r]\ar@/^/@{->}[u] & \mathbf{RE}_\cA\ar@{->}[r]\ar@/^/@{->}[u] & \mathbf{E}_\cA\ar@/^/@{->}[u]}}
\caption{Class inclusions for separation subclasses}
\label{F:LincD}
\end{figure}
\end{prop}
\begin{proof}
The horizontal arrows come straight from the definitions, and the lack of backward arrows is also straightforward. The downward arrows come from Corollary \ref{C:main} and Theorem \ref{T:same}, and the upward arrows come from Proposition \ref{P:AxSep}. The lack of an arrow from $\mathbf{FS}_\cA$ to $\mathbf{FE}_\cA$ comes from Lemma \ref{L:notFin}.
\end{proof}

Now we present some easy results on the decision problem for separation subclasses.
\begin{defn}[Subclass decision problem]
Given classes $\cA$ and $\cB$ with $\cB\subseteq \cA$, the decision problem for $\cB$ relative to $\cA$ is the question:
\[\text{``Given a finite $A\in \cA$, is $A\in \cB$?"}\]
\end{defn}

\begin{prop}\label{P:semi}
Let $\cA$ be a class of $\sL$-structures, and let $\cB$ be an essentially r.e. separation subclass of $\cA$. Then the complement to the decision problem for $\cB$ relative to $\cA$ is semidecidable.
\end{prop}
\begin{proof}
If $\cA$ is elementary then $\cB$ has a recursive axiomatisation relative to $\cA$, by Corollary \ref{C:main}, and the result follows immediately. 

Suppose now that $\cA$ is not elementary, and that $\cB$ is defined by the essentially r.e. separation scheme $\Sigma$. Our algorithm is as follows. Given finite $A\in\cA$, using dovetailing we work through the separation rules of $\Sigma$. If $\sigma$ is finite then it can be checked directly if $A\models \sigma$. If on the other hand 
\[\sigma=\forall\vec{x}_N\Big(\mu(\vec{x}_N)\rightarrow \exists \vec{C}_K(\eta(\vec{x}_N)\wedge \bw_I\tau_i)\Big)\] 
is infinite then it cannot be checked directly if $A\models \sigma$, but, as $\Sigma$ is essentially r.e., we can assume without loss of generality that $I=\omega$, and for each $n\in \omega$ we can define 
\[\sigma_n =\forall\vec{x}_N\Big(\mu(\vec{x}_N)\rightarrow \exists \vec{C}_K(\eta(\vec{x}_N)\wedge \bw_{i\leq n}\tau_i)\Big).\]
Now with dovetailing we can check if $A\models \sigma_n$ for each $n\in\omega$ and each infinite $\sigma \in \Sigma$. If $A\not\models \sigma$ then there is $n$ with $A\not\models \sigma_n$, so if such a $\sigma$ exists our algorithm will eventually find it. As soon as $\sigma$ with $A\not\models \sigma$ is found the algorithm terminates, as this shows $A\notin \cB$.
\end{proof}

\begin{lemma}\label{L:NP}
If $\cA$ is a class of $\sL$-structures and $\cB$ is an essentially finite separation subclass of $\cA$, then the decision problem for $\cB$ relative to $\cA$ is in $\mathbf{NP}$.
\end{lemma}
\begin{proof}
By Lemma \ref{L:pseud}, an essentially finite separation subclass is essentially finite pseudoelementary relative to the superclass, and being essentially finite pseudoelementary is equivalent to being finitely axiomatisable in existential second-order logic. Finally, by Fagin's Theorem \cite{Fag73}, the problem of checking whether a finite structure satisfies an existential second-order sentence is in $\mathbf{NP}$.
\end{proof}

\begin{prop}\label{P:CclassesD}
Let $\cA$ be an elementary class and make the following definitions in addition to those of Proposition \ref{P:LclassesD}:
\begin{itemize}
\item[] $\mathbf{P}_\cA$ is the class of subclasses of $\cA$ whose decision problem relative to $\cA$ is in $\mathbf{P}$.  
\item[] $\mathbf{NP}_\cA$ is the class of subclasses of $\cA$ whose decision problem relative to $\cA$ is in $\mathbf{NP}$.  
\end{itemize}
Then Figure \ref{F:Pinc} represents the class inclusions that always hold, using the same system as in Figure \ref{F:LincD}, but with the addition that the existence of a full arrow in either of the places indicated by dotted arrows is equivalent to $\mathbf{P}=\mathbf{NP}$. 
\begin{figure}[htbp]
\centerline{
\xymatrix{\mathbf{FS}_\cA\ar@{->}[r]\ar@{..>}[d]_? & \mathbf{NP}_\cA\ar@{..>}@/^1pc/[dl]^?\\
\mathbf{P}_\cA\ar@{->}[ur]}}
\caption{Complexity class inclusions}
\label{F:Pinc}
\end{figure}
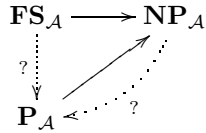
\end{prop} 
\begin{proof}
The arrow from $\mathbf{FS}_\cA$ to $\mathbf{NP}_\cA$ comes from Lemma \ref{L:NP}. To see that there is no arrow from $\mathbf{P}_\cA$ to $\mathbf{FS}_\cA$ let $\cA$ be, for example, the class of all sets, and let $\cB$ be the class of all finite sets. Then the decision problem for $\cB$ relative to $\cA$ is trivially in $\mathbf{P}$ (as every instance is a yes instance), but finiteness has no first-order characterisation, and thus cannot be formalised as a separation subclass (by Theorem \ref{T:same}). That there is no arrow from $\mathbf{NP}_\cA$ to $\mathbf{FS}_\cA$ follows immediately.  

If there is an arrow from $\mathbf{FS}_\cA$ to $\mathbf{P}_\cA$ then, for example, deciding whether a finite poset is $(4,4)$-representable is in $\mathbf{P}$, as the $(4,4)$-representable posets are an essentially finite separation subclass of the class of posets (see Example \ref{E:poset2}), and thus $\mathbf{P}=\mathbf{NP}$, as this problem is $\mathbf{NP}$-complete \cite{VanA16}. Conversely, if $\mathbf{P}=\mathbf{NP}$ then there is an arrow from $\mathbf{NP}_\cA$ to $\mathbf{P}_\cA$, and thus an arrow from $\mathbf{FS}_\cA$ to $\mathbf{P}_\cA$.  
\end{proof}

We have established that every essentially finite separation subclass of an elementary class is recursively axiomatisable relative to the superclass, and also that the converse does not hold in general (see Proposition \ref{P:LclassesD}). The following result says that the converse still doesn't hold when we restrict to varieties and recursively axiomatisable subvarieties.

\begin{prop}
There is a finitely axiomatised variety $\cA$, and a recursively axiomatised subvariety $\cB$ of $\cA$ such that $\cB$ is not an essentially finite separation subclass of $\cA$.
\end{prop}
\begin{proof}
The class $\mathbf{RRA}$ of representable relation algebras is a variety (by \cite{Tar55}, or see \cite[Theorem 3.37]{HirHod02}), and can be recursively axiomatised by equations (see \cite[Theorem 8.4]{HirHod02}), but the decision problem for $\mathbf{RRA}$ relative to the class of relation algebras, $\mathbf{RA}$, is not decidable (by \cite{HirHod01}, or see \cite[Theorem 18.13]{HirHod02}), and thus cannot be an essentially finite separation subclass (appealing to Lemma \ref{L:NP}).
\end{proof}

Finally, as in the argument used in the proof of Lemma \ref{L:NP}, an essentially finite separation subclass $\cB$ of a basic elementary class $\cA$ can be finitely axiomatised in existential second-order logic. Of course, it follows immediately from Fagin's Theorem  and Proposition \ref{P:CclassesD} that there are subclasses that are finitely axiomatisable in existential second-order logic relative to their superclasses that cannot be expressed as essentially finite separation subclasses.

\section{Applications}\label{S:Apps}
In this section we use the general theory of separation subclasses to get some axiomatisation results in graph theory and theoretical computer science.
\subsection{Disjoint union partial algebras}\label{S:dupa}
Here we deal with a class of structures introduced in \cite{HirMcL17}.
\begin{defn}
A \textbf{partial algebra} is a set equipped with a number of partial operations of fixed arities, and also possibly some constants. In order to accommodate this in first-order logic we think of partial algebras as relational structures, where each $n$-ary partial operation corresponds to an $(n+1)$-ary relation, and for each such relation $R$ we have a sentence
\[\forall x_1\ldots x_nyz\big((R(x_1,\ldots,x_n,y)\wedge R(x_1,\ldots,x_n,z))\rightarrow y\approx  z\big)\]
expressing that the associated partial function is well defined.  
\end{defn}

\begin{defn}
A \textbf{disjoint union partial algebra (DUPA)} is a partial algebra with a single ternary relation $\mathsf{d}$ (disjoint union). We will usually write $\mathsf{d}(x,y,z)$ as $x\dot{\cup}y = z$. 
\end{defn}

For more on DUPAs and their uses in computer science see \cite{HirMcL17}.

\begin{defn}[Representable DUPA]
A DUPA is \emph{representable} if it is isomorphic (as a relational structure) to a DUPA whose universe is a set of sets, and whose relation $\mathsf{d}$ is defined by
\[\mathsf{d}(X,Y,Z)\iff X\cap Y= \emptyset \text{ and } Z = X\cup Y.\]
\end{defn}
\
The following is a minor adaptation of \cite[Definition 4.1]{HirMcL17}

\begin{defn}[Basic sets]
If $A$ is a disjoint union partial algebra, and if $\Gamma\subseteq A$, then we say $\Gamma$ is \emph{basic} if:
\begin{enumerate}[(1)]
\item $a\dot\cup b\in\Gamma \implies$ either  $a\in \Gamma$ or $b\in \Gamma$.
\item If either $a\in \Gamma$ or $b\in \Gamma$, and if $a\dot\cup b$ is defined, then $a\dot\cup b\in \Gamma$.
\item If both $a\in \Gamma$ and $b\in\Gamma$ then $a\dot\cup b$ is not defined. 
\end{enumerate} 
\end{defn}

The word `basic' here refers to being part of the `base' of a representation as an algebra of sets. 

\begin{lemma}\label{L:dju}
If $A$ is a DUPA, then $A$ is representable if and only if:
\begin{enumerate}[(1)]
\item For all $a\neq b \in A$, there is a basic $\Gamma\subseteq A$ with either $a\in \Gamma$ and $b\not\in \Gamma$, or $b\in \Gamma$ and $a\not\in \Gamma$.  
\item For all $a, b \in A$, if $a\dot\cup b$ is undefined then there is a basic $\Gamma\subseteq A$ with $\{a,b\}\subseteq\Gamma$.
\end{enumerate}
\end{lemma}
\begin{proof}
This is \cite[Lemma 4.2]{HirMcL17}.
\end{proof}

\begin{prop}\label{P:dju}
The class of representable DUPAs is an essentially finite separation subclass of the class of all DUPAs.
\end{prop}
\begin{proof}
Let $\sL = \{\mathsf{d}\}$ be the signature of disjoint union partial algebras, and let $\sL^+=\sL\cup\{C\}$, where $C$ is a unary predicate symbol. Define the following $\sL^+$-sentences:
\begin{align*}\tau_0 = \forall y_1 y_2y_3&\Big(\mathsf{d}(y_1,y_2,y_3)\ra \big(C(y_3)\ra (C(y_1)\vee C(y_2))\big)\Big)  \\
\tau_1= \forall y_1 y_2y_3&\Big(\mathsf{d}(y_1,y_2,y_3)\ra\big((C(y_1)\vee C(y_2))\ra C(y_3)\big)\Big)\\
\tau_2 = \forall y_1 y_2y_3&\Big(\mathsf{d}(y_1,y_2,y_3)\ra\big(\neg C(y_1)\vee \neg C(y_2)\big) \Big).\end{align*} 
Then $\tau = \tau_0\wedge \tau_1\wedge \tau_2$ states that the set defined by $C$ is basic, and, moreover, $\tau$ is a closure rule as defined in Definition \ref{D:cRule}. Now define
\[\sigma_1 = \forall x_1x_2\Big( \neg(x_1\approx x_2)\ra \exists C \Big( \big((C(x_1)\wedge \neg C(x_2)) \vee (C(x_2)\wedge \neg C(x_1)) \big) \wedge \tau \big)\Big), \]
and  
\[\sigma_2 = \forall x_1x_2\Big( \neg\exists x_3\mathsf{d}(x_1,x_2,x_3)\ra \exists C \big( C(x_1)\wedge C(x_2) \wedge \tau \big) \Big). \]
Then $\sigma_1$ and $\sigma_2$ are separation rules, as defined in Definition \ref{D:Srule}. Moreover, by Lemma \ref{L:dju}, $\{\sigma_1,\sigma_2\}$ axiomatises the class of  representable DUPAs relative to the class of all DUPAs, which is what we are required to prove. 
\end{proof}

Having established that the class of representable DUPAs is an essentially finite separation subclass of the class of all DUPAs (which is basic elementary), we can use general results for separation subclasses to easily prove some results that were obtained with more effort in \cite{HirMcL17}. For example:

\begin{cor}
The class of representable DUPAs is basic pseudoelementary.
\end{cor}
\begin{proof}
This follows from Proposition \ref{P:dju} and Lemma \ref{L:pseud}. 
\end{proof}

\begin{cor}
The class of representable DUPAs has a recursive axiomatisation in first-order logic.
\end{cor}
\begin{proof}
This follows from Proposition \ref{P:dju} and Theorem \ref{T:main}.
\end{proof}

Note that the appearance of $\neg\exists$ in $\sigma_2$ means the recursive axiomatisation generated is not universal. Indeed, by \cite[Corollary 3.3]{HirMcL17} we know that no such universal axiomatisation can exist.

\begin{cor}
The decision problem for the class of representable DUPAs is in $\mathbf{NP}$.
\end{cor}
\begin{proof}
This follows from Lemma \ref{L:NP}.
\end{proof}

\subsection{$N$-colourable graphs}\label{S:N-colour}
Here and elsewhere we assume all graphs are undirected and simple. Given $\leq N<\omega$, a graph $G = (V,E)$ is $N$-colourable if it is possible to assign to each vertex $v\in V$ one of $N$ colours in such a way that no adjacent vertices have the same colour. Equivalently, $G$ is $N$-colourable if there is a homomorphism $h:G\to K_N$ where $K_N$ is the complete graph with $N$ vertices. Let $\cG$ be the (elementary) class of all graphs, and, given $1\leq N<\omega$, define $\cG_N$ to be the class of $N$-colourable graphs. Note that if $G$ is $N$-colourable via $h:G\to K_N$, and if $H$ is any other graph, then the composition of $h$ with the projection function, $h\circ \pi_G:G\times H\to K_N$, is a homomorphism. So, in particular, $\cG_N$ is closed under taking direct products for all $1\leq N<\omega$.   

Let $\sL =\{E\}$ be the standard signature for graphs (so $E$ stands for the binary edge relation). Let $\sL^+ = \sL\cup\{C_1,\ldots,C_N\}$ (here $C_n$ is a monadic predicate intended to pick out the vertices coloured by the $n$th colour), and define 
\[\tau_0 = \forall y \Big(\top \ra \bv^N_{n=1}C_n(y) \Big),\]
\[\tau_1 = \forall y \Big(\top \ra \bw^N_{m\neq n} \neg\big(C_m(y)\wedge C_n(y)\big) \Big),\]
\[\tau_2 = \forall y_1 y_2\Big( E(y_1,y_2) \ra \bw_{n=1}^N \neg\big(C_n(y_1)\wedge C_n(y_2)\big) \Big),\]
and
\[\sigma = \forall x \Big(\top \ra \exists C_1\ldots C_N \big(\top \wedge \tau_0\wedge\tau_1\wedge \tau_2 \big) \Big).\]

Then $\sigma$ is a separation rule, and if $\cC_N$ is the separation subclass of $\cG$ defined by $\{\sigma\}$, then $\cC_N$ is exactly the class of all $N$-colourable graphs. Thus we see that $\cC_N$ has the various pleasant properties associated with essentially finite separation subclasses of elementary classes. In particular, from Corollary \ref{C:main} we obtain a recursive universal axiomatisation for $\cC_N$ as a class of $\sL$-structures. This is not a new result. Indeed, \cite[Theorem 1.4]{Whe79} proves that $\cC_N$ has a recursive axiomatisation using universal Horn formulas, and that paper attributes to W. Taylor a proof of the same result using the De Bruijn-Erd\H{o}s theorem for graphs (i.e. that a graph is $N$-colourable when all its finite subgraphs are) \cite{deBErd51}. 

Now, being universal, $\cC_N$ is closed under isomorphisms, substructures and ultraproducts, and, as the class is also closed under taking direct products, it follows that $\cC_N$ is a universal Horn class (see e.g. \cite[Theorem V.2.23]{BurSan81}). The universal Horn theory of $\cC_N$ must be precisely the universal Horn consequences of our recursive axiomatisation, and so is also a recursively enumerable set, and consequently defines a recursive axiomatisation using Craig's trick. Thus the prima facie stronger result of \cite{Whe79} follows easily from our version, which we got more or less for free from the general theory. Note that such a universal Horn axiomatisation is the best that can be hoped for, as for $N\geq 2$ there can be no finite axiomatisation of $\cC_N$ \cite[Theorem 1.5]{Whe79}. Of course, $\cC_1$ is just the class of totally disconnected (edgeless) graphs.

Making good on the claims in the comments at the end of Section \ref{S:ax}, we can also use our game-generated axioms to find a simple proof that $\cC_N$ is not finitely axiomatisable when $N\geq 2$. First, for each $n\geq 1$ consider the cycle graph $C_{2n+1}$, and consider also the class $\cC_2$ as a separation subclass of $\cG$. Then the number of rounds $\exists$ can guarantee to survive in the simple $(C_{2n+1},\sigma)$ game scales linearly with $\log_2 n$. Here $\exists$'s strategy is to always colour vertices consistently with their closest neighbour, and $\forall$'s best strategy is to make the maximum size of a chain of uncoloured vertices as small as possible each round - see Figure \ref{F:graph} for an illustration. Note that $\exists$ can use her strategy against any strategy used by $\forall$. The one described is optimal for him in the sense that it wins fastest. 

Now, if $\cC_2$ were finitely axiomatisable then a graph would be $2$-colourable if and only if $\exists$ could guarantee survival for a fixed finite number of rounds. By choosing $n$ large enough, $\exists$ can find a graph $C_{2n+1}$ where she does have such a strategy, but which is nevertheless not $2$-colourable, and this would result in contradiction. Thus the axiomatisation of $\cC_2$ generated by Corollary \ref{C:main} cannot be logically equivalent to a finite subset of itself, and it follows that $\cC_2$ is not finitely axiomatisable. Note that as $C_{2n+1}\in \cC_3$ for all $n$, this argument also shows that $\cC_2$ is not finitely axiomatisable relative to $\cC_3$.  

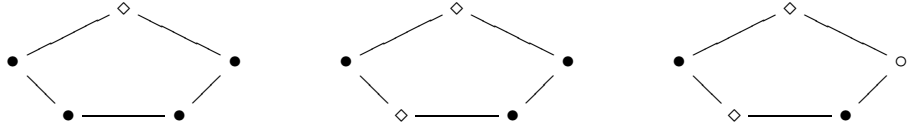
\begin{figure}[htbp]
\centering
  \begin{minipage}[b]{0.30\textwidth}
  \[
\xymatrix@=1em{&&\diamond\ar@{-}[drr]\ar@{-}[dll]\\
\bullet\ar@{-}[dr] &&&& \bullet\ar@{-}[dl] \\
&\bullet\ar@{-}[rr] & & \bullet}\] 
  \end{minipage}
  \hfill
  \begin{minipage}[b]{0.30\textwidth}
  \[
\xymatrix@=1em{&&\diamond\ar@{-}[drr]\ar@{-}[dll]\\
\bullet\ar@{-}[dr] &&&& \bullet\ar@{-}[dl] \\
&\diamond\ar@{-}[rr] & & \bullet}\]
  \end{minipage}
\hfill
  \begin{minipage}[b]{0.30\textwidth}
  \[
\xymatrix@=1em{&&\diamond\ar@{-}[drr]\ar@{-}[dll]\\
\bullet\ar@{-}[dr] &&&& \circ\ar@{-}[dl] \\
&\diamond\ar@{-}[rr] & & \bullet}\]
  \end{minipage}
\caption{A game played on $C_5$, with different `colours' being denoted by $\diamond$ and $\circ$. Here $\forall$ first asks $\exists$ to colour the top vertex, which she does with $\diamond$ (both players making arbitrary choices here). To minimise the maximum length of a chain of uncoloured elements $\forall$ then demands that one of the bottom two vertices be coloured (the left one say). To be consistent with the closest coloured vertex $\exists$ responds by colouring with $\diamond$. Now $\forall$ demands that either of the vertices between the two coloured ones (moving clockwise) be coloured (the higher one, say). To be consistent with the nearest neighbour $\exists$ must colour it with $\circ$. Now $\forall$ can force a forbidden colouring in the next round.}
\label{F:graph}
\end{figure}

Generalising, let $N>2$, and for each $n\geq 1$ define $G_n$ to be the graph obtained by taking the disjoint union of the cycle graph $C_{2n+1}$ and the complete graph $K_{N-2}$, and adding edges between every vertex of $C_{2n+1}$ and every vertex of $K_{N-2}$. Then, in the game where $\exists$ attempts to colour $G_n$ using $N$ colours, the choice of colours for $K_{N-2}$ forces her to attempt to colour $C_{2n+1}$ with two colours. We know this is impossible, but the number of rounds she can survive again scales with $\log_2n$. Here $\exists$'s strategy is to choose $N-2$ colours for $K_{N-2}$, and to use her strategy from the $N=2$ case for $C_{2n+1}$ with the two remaining colours. Thus $\cC_N$ is not finitely axiomatisable for all $N\geq 2$. This provides a proof of \cite[Theorem 1.5]{Whe79} that does not use the fact that the class of graphs with chromatic number $N$ is not elementary for all $N\geq 3$ \cite[Theorem 6.3]{Tay69}. Note that, combined with the result for $N= 2$, this argument also shows that $\cC_{N}$ is not finitely axiomatisable relative to $\cC_{N+1}$ for all $N\geq 2$.

Moreover, let $N\geq 3$ and define $\chi_N$ to be the class of graphs with chromatic number $N$. We can use our results on the lack of a finite axiomatisation for $\cC_{N-1}$ relative to $\cC_N$ to prove that $\chi_N$ is not elementary. To see this, first note that $\chi_N = \cC_N\setminus \cC_{N-1}$, and so $\cC_N = \cC_{N-1}\cup \chi_N$, which is a disjoint union. Now, as $\cC_N$ and $\cC_{N-1}$ are elementary, if $\chi_N$ is also elementary then $\cC_{N-1}$ will be finitely axiomatisable relative to $\cC_N$, by a variation of the compactness argument that says that if a class and its complement are elementary, then both will be basic elementary. As $\cC_{N-1}$ is not finitely axiomatisable relative to $\cC_N$, it follows that $\chi_N$ is not elementary. Thus we also obtain an alternative proof of \cite[Theorem 6.3]{Tay69} (the original uses Erd\H{o}s' famous result that for all $m,k\in\omega$ there is a finite graph with chromatic number $\geq m $ and no circuits of length $\leq k$ \cite{Erd59}).  

As a final observation, every first-order structure can be embedded into an ultraproduct of its finitely generated substructures (see e.g. \cite[Theorem V.2.14]{BurSan81}). Moreover, if a graph $G$ has the property that every finite subgraph is $N$-colourable, then, as $\cC_N$ is elementary, an ultraproduct of these subgraphs must also be $N$-colourable, by \L o\'s' theorem. Furthermore, as $\cC_N$ is universal, its substructures must also be in $\cC_N$, and so it follows that $G\in \cC_N$. Thus, from the axiomatisation of $\cC_N$ we also obtain a rather indirect proof of the De Bruijn-Erd\H{o}s theorem. We must note that much simpler proofs are well known, so this last result is essentially a curiosity.  

\subsection{Clique covers}\label{S:clique}   
Let $N\in\omega$. We say a graph $G = (V,E)$ has an $N$-clique cover if its vertices can be partitioned into $N$ subsets, each of which is a clique. In other words, if there is a partition $V_1,\ldots,V_N$ of $V$ such that the restriction of $E$ to $V_n$ produces a complete graph for all $n\in \{1,\ldots,N\}$. Note that a graph $G$ has an $N$-clique covering if and only if the complement graph $\bar{G} = (V, \bar{E})$ is $N$-colourable. As we can define $\bar{E}$ as $\neg E$, the results of Section 6.2 apply here, with the following exception. The class of undirected simple graphs with an $N$-clique cover is \emph{not} closed under taking direct products for any $N$. To see this, consider the product of the totally disconnected graph with $N$ vertices with itself. So this class does not have a universal Horn axiomatisation (by \cite[Theorem V.2.23]{BurSan81} again), though it does have a recursive universal axiomatisation. 

\subsection{Harmonious colourings}\label{S:harm}
The concept of a harmonious colouring for a graph was introduced in \cite{FHP82} and defined in its current form in \cite{HopKri83}. Given $N\in\omega$, we say a graph has a harmonious $N$-colouring if it has an $N$-colouring in which each pair of colours can be used to colour a pair of adjacent vertices at most once.  I.e. if we use $c(v)$ to denote the colour of a vertex, and if $v_1,v_2,v_3,v_4$ are vertices with $c(v_1)=c(v_3)$ and $c(v_2)=c(v_4)$, then either $v_1=v_3$ and $v_2=v_4$, or there cannot be edges between both $\{v_1,v_2\}$ and $\{v_3,v_4\}$. Define $\sL^+$, $\tau_0$, $\tau_1$ and $\tau_2$ as in Section \ref{S:N-colour}. In addition, we define a sentence $\tau_3$ expressing that the vertices of two distinct edges cannot be coloured by the same pattern of colours. This sentence says that if $E(v_1,v_2)$ and $E(v_3,v_4)$ (i.e. the edges exist), and if $v_1\neq v_3$ or $v_2\neq v_4$ (i.e. the edges are distinct), then either $v_1$ is not the same colour as $v_3$, or $v_2$ is not the same colour as $v_4$. 
 
\begin{align*}\tau_3 = \forall y_1 y_2 y_3 y_4\Big(&\big(\neg((y_1\approx y_3) \wedge(y_2\approx y_4))\wedge E(y_1,y_2)\wedge E(y_3,y_4)\big) \\ 
&\ra \bw_{m,n=1}^N \neg \big(C_m(y_1)\wedge C_n(y_2) \wedge C_m(y_3)\wedge C_n(y_4)\big) \Big).\end{align*}
Now define 
\[\sigma = \forall x \Big(\top \ra \exists C_1\ldots C_n \big(\top \wedge \tau_0\wedge\tau_1\wedge \tau_2\wedge \tau_3 \big) \Big).\]
Then $\sigma$ defines the class of graphs with harmonious $N$-colourings as an essentially finite separation subclass of $\cG$. It again follows from Corollary \ref{C:main} that this class has a universal recursive axiomatisation. Note that when $N\geq 2$ the class is not closed under taking direct products. To see this, note that a graph with a harmonious $N$-colouring can have at most ${N \choose 2}$ edges, as this is the maximum number of distinct colour pairs, and consider the product of the complete graphs $K_N$ and $K_2$. Each component has a harmonious $N$-colouring, but the product does not, simply because it has too many edges. So the class does not have a universal Horn axiomatisation. When $N=1$ the graphs must be totally disconnected just to have an $N$-colouring, which will be trivially harmonious.

From the fact that a graph with a harmonious $N$-colouring can have at most ${N \choose 2}$ edges, we can deduce that the axiomatisation produced here is equivalent to a finite one. Assuming $\forall$ plays in an efficient way, in other words, that he forces $\exists$ to define a new coloured pair each round if possible, he will either definitely be able to force a win in round $({N\choose 2} +1)$ at the latest, or he will have run out of useful moves in an earlier round. So, if $\exists$ has an $({N\choose 2} +1)$-strategy, then she has an $\omega$-strategy. Appealing to Lemma \ref{L:finite} proves the claim.

Note that it is proved in \cite{HopKri83} that the problem of deciding, when given a graph $G$ and a positive integer $N$, whether $G$ has a harmonious colouring with $N$-colours is $\mathbf{NP}$-complete. As it is known that checking whether a first-order sentence is valid in a finite structure can be done in polynomial time (see \cite[Proposition 3.1]{Var95}), we may wonder whether we have accidentally proved $\mathbf{P}=\mathbf{NP}$. The answer, sadly, is no, because given $(G,N)$ we have to \emph{construct} the appropriate sentence, which depends on $N$, before we can check it, and we have no reason to believe we can do this in polynomial time. There is an alternative version of the problem where $N$ is regarded as fixed, and so an instance is just a graph $G$. In this case our argument does indeed show the decision problem to be in $\mathbf{P}$, but this version of the problem is not $\mathbf{NP}$-complete.      

\section*{Acknowledgment}
The author would like to thank Robin Hirsch for, among other things, a suggestion that significantly simplified the key definitions, and the observation described in the preceding paragraph. The author would also like to thank the Department of Computer Science at UCL for hosting him while part of this paper was written. In addition, the author thanks the anonymous referee for their useful comments.  
  
%\bibliography{../../../../references}{}
\bibliographystyle{abbrv}

\end{document}